\documentclass[12pt,leqno]{article}

\usepackage{amssymb,amsfonts,amsmath,amsthm,amscd,mathrsfs}
\usepackage{graphics,graphicx}
\def\IC{{\mathbb C}}
\def\IE{{\mathbb E}}
\def\IL{{\mathbb L}}
\def\IP{{\mathbb P}}
\def\IR{{\mathbb R}}

\def\IZ{{\mathbb Z}}

\def\n{\noindent}
\def\dsl{\textstyle\sum\limits}

\def\dis{\displaystyle}
\def\o{\omega}
\def\fr{\mbox{\footnotesize $\dis\frac{1}{2}$}}

\def\ov{\overline}
\def\f{\footnotesize}
\def\r{\rightarrow}
\def\point{{\mbox{\large $.$}}}
\def\wh{\widehat}
\def\wt{\widetilde}

\def\ve{\varepsilon}

\def\cA{{\cal A}}
\def\cB{{\cal B}}

\def\cL{{\cal L}}
\def\cT{{\cal T}}

\def\cI{{\cal I}}

\def\cF{{\cal F}}

\def\cS{{\cal S}}

\def\cV{{\cal V}}
\def\cW{{\cal W}}

\dimendef\dimen=0

\newtheorem{theorem}{Theorem}[section]
\newtheorem{lemma}[theorem]{Lemma}

\newtheorem{corollary}[theorem]{Corollary}
\newtheorem{proposition}[theorem]{Proposition}
\newtheorem{remark}[theorem]{Remark}

\begin{document}

\baselineskip14pt
\noindent

\begin{center}
{\bf ON SCALING LIMITS AND BROWNIAN INTERLACEMENTS}
\end{center}

\vspace{0.5cm}

\begin{center}
Alain-Sol Sznitman$^*$
\end{center}

\bigskip
%\begin{center}
%Preliminary Draft
%\end{center}

\bigskip
\begin{abstract}
We consider continuous time interlacements on $\IZ^d$, $d\ge 3$, and investigate the scaling limit of their occupation times. In a suitable regime, referred to as the {\it constant intensity regime}, this brings Brownian interlacements on $\IR^d$ into play, whereas in the {\it high intensity regime} the Gaussian free field shows up instead. We also investigate the scaling limit of the isomorphism theorem of \cite{Szni12b}. As a by-product, when $d=3$, we obtain an isomorphism theorem for Brownian interlacements.
\end{abstract}

\vspace{8cm}

\noindent
Departement Mathematik\\  %\hfill  September 2012\\
ETH-Zurich\\
CH-8092 Z\"urich\\
Switzerland

\vfill
\noindent
$\overline{~~~~~~~~~~~~~~~~~~~~~~~}$\\
$^*$ {\small This research was supported in part by the grant ERC-2009-AdG  245728-RWPERCRI}

\newpage

\thispagestyle{empty}
~

\newpage
\setcounter{page}{1}

\setcounter{section}{-1}
\section{Introduction}
Informally, random interlacements provide a model for the local structure left at appropriately chosen time scales by random walks on large recurrent graphs, which are locally transient. Their vacant set has non-trivial percolative properties, see for instance \cite{DrewRathSapo12}, \cite{SidoSzni09}, \cite{Szni10a}, \cite{Teix11}, which are useful in the study of certain disconnection and fragmentation problems, see~\cite{CernTeix12}, \cite{CernTeixWind11}, \cite{Szni09b}, \cite{TeixWind11}. Their connectivity properties have been actively investigated, see \cite{CernPopo12}, \cite{LacoTyke12}, \cite{ProcTyke11}, \cite{RathSapo12}, \cite{RathSapo11a}. Random interlacements have further been helpful in some questions of cover times, see~\cite{Beli11}, \cite{Beli12c}. They have also been linked to Poisson gases of Markovian loops, see~\cite{Leja12}, \cite{Szni11d}. Recently, connections between random interlacements and Gaussian free fields have emerged, which underline the important role of occupation times of random interlacements, see~\cite{Szni12b}.

\medskip
In this article, we investigate the scaling limit of the field of occupation times of continuous time interlacements on $\IZ^d$, $d \ge 3$. In the {\it constant intensity regime}, this brings into play the Brownian interlacements on $\IR^d$, $d \ge 3$. In the {\it high intensity regime}, the massless Gaussian free field shows up instead. Further, we investigate the scaling limit of the isomorphism theorem derived in \cite{Szni12b}. Dimension three plays a special role, and when $d=3$, we obtain as a limit an identity in law relating the occupation-time measure of Brownian interlacements on $\IR^3$ to the massless Gaussian free field, somewhat in the spirit of \cite{Leja10}, \cite{Leja12}, in the context of Poisson gases of Brownian loops at half-integer levels.

\medskip
We now discuss our results in more detail. We consider continuous time random interlacements on $\IZ^d$, $d \ge 3$. In essence, this is a Poisson point process on a certain state space consisting of doubly infinite $\IZ^d$-valued trajectories marked by their duration at each step, modulo time-shift. A non-negative parameter $u$ plays the role of a multiplicative factor of the intensity measure of this Poisson point process, which is defined on a suitable canonical space, see \cite{Szni11c}, denoted here by $(\ov{\Omega},\ov{\cA},\ov{\IP})$. The field of occupation times of random interlacements at level $u$ is denoted by $L_{x,u}(\o)$, for $x\in \IZ^d$, $u \ge 0$, $\o \in \ov{\Omega}$. It records the total duration spent at $x$ by the trajectories modulo time-shift with label at most $u$ in the cloud $\o$, see \cite{Szni11c}.

\medskip
We investigate the scaling limit for the field of occupation times of random interlacements on $\IZ^d$, $d \ge 3$, and introduce the random measure on $\IR^d$
\begin{equation}\label{0.1}
\cL^N = \mbox{\f $\dis\frac{1}{dN^2}$} \; \dsl_{x \in \IZ^d} L_{x,u_N} \,\delta_{\frac{x}{N}}, \; N \ge 1,
\end{equation}
where $(u_N)_{N \ge 1}$ is a suitably chosen sequence of positive levels. In the {\it constant intensity regime}, that is when
\begin{equation}\label{0.2}
u_N = d \alpha  N^{2-d}, \;\; \mbox{with} \; \alpha > 0
\end{equation}

\n
(in this case the intensity measure $\ov{\IE}[\cL^N]$ of $\cL^N$ converges vaguely to $\alpha \, dy$ as $N$ goes to infinity), we show in Theorem \ref{theo3.2} that
\begin{equation}\label{0.3}
\mbox{$\cL^N$ converges in distribution to $\cL_\alpha$, as $N \r \infty$},
\end{equation}
where $\cL_\alpha$ stands for the occupation-time measure of Brownian interlacements at level $\alpha$, see (\ref{2.37}). We construct Brownian interlacements in Section 2 using a similar strategy as in \cite{Szni10a} or \cite{Teix09b}. This is technically somewhat more involved in the present context, and the key identity is encapsulated in Lemma \ref{lem2.1}. The random measure $\cL_\alpha$ is supported by the random closed set $\cI_0^\alpha$, the {\it Brownian fabric at level $\alpha$}, which is the union of the traces in $\IR^d$ of the doubly infinite trajectories modulo time-shift with label at most $\alpha$ in the Poisson cloud defining the Brownian interlacements, see (\ref{2.30}). This random closed subset of $\IR^d$ is a.s. connected when $d=3$, but a.s. disconnected when $d \ge 4$, see Proposition \ref{prop2.5}.

\medskip
On the other hand in the {\it high intensity regime}, that is when
\begin{equation}\label{0.4}
N^{d-2} u_N \r \infty ,
\end{equation}

\n
we show in Theorem \ref{theo3.3} and Corollary \ref{cor3.5} that
\begin{align}
\wh{\cL}^N = \sqrt{\mbox{\f $\dis\frac{d}{2N^{2-d}u_N}$}} \;\big(\cL^N - \ov{\IE}[\cL^N]\big) & \; \mbox{converges in distribution to the} \label{0.5}
\\[-2ex]
&\; \mbox{massless Gaussian free field $\Phi$ on $\IR^d$}, \nonumber
\end{align}

\n
which is the canonical generalized random field on $\cS'(\IR^d)$, the space of tempered distributions on $\IR^d$, such that for any $V$ in the Schwartz space $\cS(\IR^d)$, the pairing $\langle \Phi, V\rangle$ is a centered Gaussian variable with variance $\int V(y) G(y-y') V(y')dy dy'$, with $G(\cdot)$ the Green function of Brownian motion on $\IR^d$, $d \ge 3$, see (\ref{1.3}).

\medskip
The {\it low intensity regime}, when $N^{d-2} u_N \r 0$, leads to a null limit for $\cL^N$ in (\ref{0.1}), and will not be further discussed in the present work.

\medskip
From the isomorphism theorem for random interlacements, see \cite{Szni12b}, one knows that when $(\varphi_x)_{x \in \IZ^d}$ is a discrete massless Gaussian free field, i.e.~a centered Gaussian field with covariance $E[\varphi_x\,\varphi_{x'}] = g(x-x')$, with $g(\cdot)$ the Green function of simple random walk on $\IZ^d$, see (\ref{1.1}), independent from $(L_{x,u})_{x \in \IZ^d}$, one has for any $u \ge 0$ the distributional identity:
\begin{equation}\label{0.6}
\Big(\fr \;\varphi^2_x + L_{x,u}\Big)_{x \in \IZ^d} \stackrel{\rm law}{=} \Big(\fr \,(\varphi_x + \sqrt{2u})^2\Big)_{x \in \IZ^d}.
\end{equation}

\n
We explain in Section 4 how one quickly recovers (\ref{0.5}) from this identity in the regime where, in essence, as $N \r \infty$, the variance of $\sum_{|x| \le N} \varphi_x^2$ is negligible compared to that of $\sqrt{u}_N$ $\sum_{|x| \le N} \varphi_x$. This corresponds to the full range (\ref{0.4}) when $d = 3$, but only the partial range of (\ref{0.4}) where $u_N N^2 \log N \r \infty$, when $d = 4$, and $u_N N^2 \r \infty$, when $d \ge 5$, see Remark \ref{rem4.2}. This fact singles out the special role of dimension 3, and motivates the investigation when $d=3$ of the scaling limit of (\ref{0.6}) (with adequate counter terms) in the constant intensity regime (\ref{0.2}). Indeed, when $d = 3$, Theorem \ref{theo5.1} roughly states that for any $\alpha \ge 0$ (denoting Wick products by {\large : \!\!:}, see Section 5)
\begin{equation}\label{0.7}
\mbox{\f $\dis\frac{1}{dN^2}$} \;\dsl_{x \in \IZ^d} \mbox{{\large :}} \big(\varphi_x + \sqrt{2u_N}\big)^2\mbox{{\large :}} \; \mbox{$\delta_{\frac{x}{N}}$ converges in distribution to {\large :}$\big(\Phi + \sqrt{2 \alpha}\big)^2${\large :}}\,,
\end{equation}
with $u_N$ as in (\ref{0.2}), and the last term defined by regularization of the massless Gaussian free field $\Phi$ on $\IR^3$ (crucially using the fact that the Green function is locally square integrable when $d = 3$), see the beginning of Section 5. As a scaling limit of (\ref{0.6}) we obtain the identity in law on $\cS'(\IR^3)$, for $\alpha \ge 0$:
\begin{equation}\label{0.8}
\fr\; \mbox{{\large :}} \Phi^2\mbox{{\large :}} + \cL_\alpha = \fr \;\mbox{{\large :}}\big(\Phi + \sqrt{2 \alpha}\big)^2\mbox{{\large :}}\,,
\end{equation}

\n
where $\cL_\alpha$ is independent of $\Phi$ and denotes, as above, the occupation-time measure of Brownian interlacements on $\IR^3$ at level $\alpha$, see Corollary \ref{cor5.3}. This identity in law has a similar spirit to some of the results in \cite{Leja10}, \cite{Leja12} in the context of Poisson gases of Brownian loops at half-integer levels, see Remark \ref{rem5.4}.

Let us describe how the article is organized. In Section 1 we collect some useful facts concerning Green functions, occupation times of random interlacements on $\IZ^d$, and discrete as well as continuous free fields. Section 2 is of independent interest. It constructs the Brownian interlacements on $\IR^d$, $d \ge 3$, and derives some of their key properties, see Propositions \ref{prop2.4}, \ref{prop2.5}, \ref{prop2.6}. In Section 3, we investigate the limit of $\cL^N$, respectively $\wh{\cL}^N$, in the constant intensity, respectively, high intensity regime. The main results appear in Theorems \ref{theo3.2}, \ref{theo3.3} and Corollary \ref{cor3.5}. In the short Section 4 we use the isomorphism theorem (\ref{0.6}) as a mean to recover (\ref{0.5}). Section 5 focuses on the three-dimensional situation. The scaling limit of (\ref{0.6}) (with adequate counter terms) is investigated in the constant intensity regime (\ref{0.2}). The main results appear in Theorem \ref{theo5.1} and Corollary \ref{cor5.3}.

\medskip
Finally, let us explain our convention concerning constants. We denote by $c,c',\wt{c},\ov{c}$ positive constants changing from place to place, which simply depend on $d$. Numbered constants $c_0,c_1,\dots$ refer to the value corresponding to their first appearance in the text. Dependence of constants on additional parameters appears in the notation.

\section{Some useful facts}
\setcounter{equation}{0}

In this section, we introduce additional notation, and collect some useful facts concerning Green functions, occupation times of random interlacements on $\IZ^d$, $d \ge 3$, and massless Gaussian free fields on $\IZ^d$ and $\IR^d$, $d \ge 3$.

\medskip
We write $|\cdot|$, respectively $|\cdot|_\infty$, for the Euclidean, respectively, the supremum norm on $\IR^d$. Throughout, we tacitly assume $d \ge 3$. We let $B(y,r)$, $y \in \IR^d$, $r \ge 0$, stand for the closed Euclidean ball with center $y$ and radius $r$.

\medskip
We denote by $g(\cdot,\cdot)$  the Green function of simple random walk on $\IZ^d$, that is, for $x,x' \in \IZ^d$, $g(x,x')$ is the expected time spent at $x'$ for the discrete time, simple random walk starting at $x$. The function $g(\cdot,\cdot)$ is symmetric, and due to translation invariance, one has for $x,x'$ in $\IZ^d$
\begin{equation}\label{1.1}
g(x,x') = g(x-x') = g(x' - x), \; \mbox{where} \; g(\cdot) = g(\cdot,0).
\end{equation}

\n
One knows that $g(\cdot) \le g(0)$ and that when $x$ tends to infinity, cf.~\cite{Lawl91}, p.~31,
\begin{align}
&g(x) \sim d G(x), \; \mbox{where}\label{1.2}
\\[1ex]
&G(y) = \dis\frac{1}{2 \pi^{d/2}} \; \Gamma\Big(\mbox{\f $\dis\frac{d}{2} - 1\Big)$} \,|y|^{2-d}, \;\mbox{for $y \in \IR^d$}, \label{1.3}
\end{align}

\medskip\n
and ``$\sim$'' in (\ref{1.2}) means that the ratio of the two members tends to $1$ as $x$ goes to infinity. Writing $G(y,y') = G(y-y') = G(y' - y)$, one knows, see \cite{Durr84}, p.~32, that the $(0,\infty]$-valued function $G(\cdot,\cdot)$ is the Green function of Brownian motion on $\IR^d$.

\medskip
For $N \ge 1$, we denote by $\IL_N$ the lattice in $\IR^d$:
\begin{equation}\label{1.4}
\IL_N = \mbox{\f $\dis\frac{1}{N}$} \;\IZ^d.
\end{equation}

\n
For functions $f,h$ on $\IL_N$ such that $\sum_{y \in \IL_N} |f(y)h(y)| < \infty$, we write
\begin{equation}\label{1.5}
\big\langle f,h \big\rangle_{\IL_N} = \mbox{\f $\dis\frac{1}{N^d}$} \;\dsl_{y \in \IL_N} f(y) h(y) .
\end{equation}
We rescale the Green function $g(\cdot,\cdot)$ and define
\begin{equation}\label{1.6}
g_N(y,y') = \mbox{\f $\dis\frac{1}{d}$} \;N^{d-2} g(Ny,Ny'), \;\mbox{for}\; y,y' \in \IL_N,
\end{equation}

\n
as well as $g_N(\cdot) = g_N(\cdot,0)$, so that $g_N(y,y') = g_N(y-y') = g_N(y'-y)$. By (\ref{1.2}), (\ref{1.3}), we also see that
\begin{equation}\label{1.7}
\lim\limits_N \;\sup\limits_{|y| \ge \gamma}\, |g_N(y) - G(y)| = 0, \;\; \mbox{for every} \; \gamma > 0.
\end{equation}
We introduce the linear operator
\begin{equation}\label{1.8}
G_N f(y) = \mbox{\f $\dis\frac{1}{N^d}$} \; \dsl_{y' \in \IL_N} g_N(y,y') \,f(y'), \;y \in \IL_N,
\end{equation}

\n
which is well-defined when the function $f$: $\IL_N \r \IR$ is such that $\sum_{y \in \IL_N} g_N(y,y') \,|f(y')| < \infty$ for some (and hence all) $y$ in $\IL_N$, in particular, when $f$ vanishes outside a finite set.

\medskip
Similarly, when $f,h$ are measurable functions on $\IR^d$ such that $|fh|$ is integrable, we write
\begin{equation}\label{1.9}
\big\langle f,h\big\rangle = \dis\int_{\IR^d} f(y) \,h(y) \,dy.
\end{equation}
We also introduce the linear operator
\begin{equation}\label{1.10}
G f(y) = \dis\int_{\IR^d} G(y,y') \,f(y') \,dy', \; \mbox{for} \; y \in \IR^d,
\end{equation}

\n
which is well-defined when the measurable function $f$ on $\IR^d$ satisfies $\int G(y,y') \,|f(y')| \,dy' < \infty$ for all $y \in \IR^d$, in particular when $f$ is bounded measurable and vanishes outside a bounded set.

\medskip
We now recall an identity for the Laplace transform of the field of occupation times $(L_{x,u})_{x \in \IZ^d}$ of continuous time random interlacements on $\IZ^d$ at level $u \ge 0$. When $V: \IL_N \r \IR$ has finite support, the operator $G_NV$, which is the composition of the multiplication by $V$ with the operator $G_N$ in (\ref{1.8}), sends bounded functions on $\IL_N$ into bounded functions on $\IL_N$. We write $\| G_N V\|_{L^\infty \r L^\infty}$ for the corresponding operator norm (the space of bounded functions on $\IL_N$ being endowed with the sup-norm). One knows from Theorem 2.1 of \cite{Szni11c} that when $V$: $\IL_N \r \IR$ has finite support and $\|G_N \,|V|\,\|_{L^\infty \r L^\infty} < 1$, then for $u \ge 0$,
\begin{equation}\label{1.11}
\ov{\IE} \Big[\exp\Big\{\dsl_{x \in \IZ^d} V\big(\mbox{\f $\dis\frac{x}{N}$}\big) \;\mbox{\f $\dis\frac{1}{dN^2}$} \;L_{x,u}\Big\}\Big] = \exp\Big\{\mbox{\f $\dis\frac{u}{d}$} \;N^{d-2} \big\langle V,(I - G_N V)^{-1} 1\big\rangle_{\IL_N}\Big\}.
\end{equation}

\n
Note  that when $V$ vanishes outside a single point one readily finds by differentiation that:
\begin{equation}\label{1.12}
\ov{\IE}[L_{x,u}] = u, \;\mbox{for $x \in \IZ^d$ and $u \ge 0$}.
\end{equation}

\n
We then turn to the discussion of Gaussian free fields. Recall that we tacitly assume $d \ge 3$. We begin with the discrete case. We endow $\IR^{\IZ^d}$ with the product $\sigma$-algebra, and denote by $(\varphi_x)_{x \in \IZ^d}$ the canonical coordinates. The canonical law $P^g$ of the massless Gaussian free field on $\IZ^d$ is characterized by the fact that
\begin{equation}\label{1.13}
\begin{array}{l}
\mbox{under $P^g$, $(\varphi_x)_{x \in \IZ^d}$ is a centered Gaussian field with covariance}\\
E^{P^g}[\varphi_x \varphi_{x'}] = g(x,x'), \;\mbox{for} \; x,x' \in \IZ^d .
\end{array}
\end{equation}

\n
In the continuous case, we consider the space $\cS'(\IR^d)$ of tempered distributions on $\IR^d$. We endow $\cS'(\IR^d)$ with the cylindrical $\sigma$-algebra generated by the canonical pairings with functions of $\cS(\IR^d)$. We write $\Phi$ for the canonical generalized random field on $\cS'(\IR^d)$ (i.e. the identity map). The symmetric bilinear form on $\cS(\IR^d) \times \cS(\IR^d)$ defined by
\begin{equation}\label{1.14}
E(V,W) = \dis\int V(y) \,G(y-y') \,W(y') \,dy \, dy' \big( = \big\langle V,G W\big\rangle = \big\langle GV, W\big\rangle\big)
\end{equation}
is positive definite (see for instance \cite{Szni98a}, p.~75). It satisfies the bound
\begin{equation*}
|E(V,W)| \le c\,\|V\|_{L^1(\IR^d)}(\|W\|_{L^\infty(\IR^d)} + \|W\|_{L^1(\IR^d)}), \;\mbox{for} \; V,W \in \cS(\IR^d),
\end{equation*}

\n
and hence is continuous (we endow $\cS(\IR^d)$ with its usual Fr\'echet topology, see \cite{Ito84}, p.~6-8). By Minlos' Theorem, see Theorem 2.3, p.~12 of \cite{Simo79}, or Theorem 2.4.1, p.~28 of \cite{Ito84}, there exists a unique probability measure $P^G$ on $\cS'(\IR^d)$ such that
\begin{equation}\label{1.15}
\begin{array}{l}
\mbox{under $P^G$, for each $V \in \cS(\IR^d)$, $\big\langle \Phi, V\big\rangle$ is a centered Gaussian variable}
\\
\mbox{with variance $E(V,V)$.}
\end{array}
\end{equation}

\n
The law $P^G$ describes the massless Gaussian free field on $\IR^d$, see Chapter 6  \S 2 of \cite{GlimJaff81}.

\section{Brownian interlacements}
\setcounter{equation}{0}

In this section we construct the Brownian interlacements on $\IR^d$, $d \ge 3$, introduce their occupation-time measure, and discuss some key properties of these objects. We follow a similar approach as in \cite{Szni10a}, \cite{Teix09b}. However, the continuous set-up is technically more challenging. An important step is the construction of the intensity measure of our basic point process on a space of doubly infinite trajectories modulo time-shift. We  mainly rely on a crucial compatibility property of the measure expressed in ``local charts'', see Lemma 2.1 (cf.~Theorem 1.1 of \cite{Szni10a} and Theorem 2.1 of \cite{Teix09b}), rather than on an approach based on projective limits, see \cite{Weil70}, \cite{Silv74}, following the outline of \cite{Hunt60}.

\medskip
We first need to introduce notation. We recall that we tacitly assume $d \ge 3$. We denote by $W_+$ the subspace of $C(\IR_+,\IR^d)$ of continuous $\IR^d$-valued trajectories tending to infinity at infinite times, and by $W$ the subspace of $C(\IR,\IR^d)$ consisting of continuous bilateral trajectories from $\IR$ into $\IR^d$, which tend to infinity at plus and minus infinite times. We write $X_t$, $t \ge 0$, and $X_t$, $t \in \IR$, for the respective canonical processes, and denote by $\theta_t$, $t \ge 0$, and $\theta_t$, $t \in \IR$, the respective canonical shifts. The spaces $W_+$ and $W$ are endowed with the respective $\sigma$-algebras $\cW_+$ and $\cW$ generated by the canonical processes. Given an open subset $U$ of $\IR^d$, $w \in W_+$, we write $T_U(w) = \inf\{s \ge 0; X_s(w) \notin U\}$ for the exit time of $U$. When $F$ is a closed subset of $\IR^d$, we write $H_F(w) = \inf\{s \ge 0; X_s (w) \in F)\}$, $\wt{H}_F(w) = \inf\{s > 0; X_s (w) \in F\}$, for the respective entrance time and hitting time of $F$. When $w \in W$, we define $H_F(w)$ and $T_U(w)$ by similar formulas, simply replacing the condition $s \ge 0$, by $s \in \IR$.

\medskip
We then consider $W^*$ the set of equivalence classes of trajectories in $W$ modulo time-shift, i.e.~$W^* = W/\sim$, where $w \sim w'$, when $w(\cdot) = w'(\cdot + t)$ for some $t \in \IR$. We denote by $\pi^*$ the canonical map $W \r W^*$, and introduce the $\sigma$-algebra $\cW^* = \{A \subseteq W^*; (\pi^*)^{-1}(A) \in \cW\}$, which is the largest $\sigma$-algebra on $W^*$ such that  $(W, \cW) \stackrel{\pi^*}{\longrightarrow}(W^*,\cW^*)$ is measurable. Incidentally, note that a random variable $Z$ on $W$ invariant under $\theta$ (i.e.~$Z \circ \theta_t = Z$, for all $t$ in $\IR$) determines a unique random variable $Z^*$ on $W^*$ such that  $Z^* \circ \pi^* = Z$. Given a compact subset $K$ of $\IR^d$, we write $W_K$ for the subset of trajectories of $W$ that enter $K$, and $W_K^*$ for its image under $\pi^*$.

\medskip
Since $d \ge 3$, and Brownian motion on $\IR^d$ is transient, we view $P_y$, the Wiener measure starting from $y \in \IR^d$, as defined on $(W_+, \cW_+)$, and write $E_y$ for the corresponding expectation. When $\rho$ is a finite measure on $\IR^d$, we write $P_\rho$ for the Wiener measure with ``initial distribution'' $\rho$ and $E_\rho$ for the corresponding expectation.

\medskip
When $B$ is a closed Euclidean ball of positive radius (as a shorthand in what follows, we will simply write that $B$ is a closed ball), and when $y \notin B$, we define $P_y^B[\cdot] = P_y[\,\cdot \,|H_B = \infty]$ to be the law of Brownian motion starting at $y$ conditioned never to enter $B$, and write $E^B_y$ for the corresponding expectation. When $y \in \partial B$ (the boundary of $B$), as $z$ in $\IR^d \backslash B$ tends to $y$, the measures $P^B_z$ converge weakly (on $C(\IR_+,\IR^d)$) to a measure $P^B_y$ supported on $\{w \in W_+; w(0) = y$ and $w(t) \notin B$ for all $t > 0\}$, which can be represented as the Brownian excursion measure in $\IR^d \backslash B$ starting from $y$ conditioned on the event of finite positive mass that $\{\wt{H}_B = \infty\}$ (see Theorems 4.1 and 2.2 of \cite{Burd87}, and also Theorem 3.1 of \cite{CranMcco83} for bounds on the exit time from small balls centered at $y$ under $P^B_z$). Using rotational invariance and the explicit nature of the conditioning, one can also in a more direct fashion establish the above mentioned facts using a skew product representation of the law $P^B_z$.

\medskip
We write  $p_t(y,y') = \frac{1}{(2 \pi t)^{d/2}} \exp\{- \frac{|y - y'|^2}{2 t}\}$, with $t > 0$, $y, y' \in \IR^d$, for the Brownian transition density, so that with similar notation as below (\ref{1.3})
\begin{equation}\label{2.1}
G(y,y') = \dis\int^\infty_0 p_t(y,y')\,dt, \;\;  \mbox{for} \; y,y' \in \IR^d.
\end{equation}

\n
Given a compact subset $K$ of $\IR^d$, we denote by $e_K$ the equilibrium measure of $K$ (see Proposition 3.3, p.~58 of \cite{Szni98a}, or Theorem 1.10, p.~58 of  \cite{PortSton78}). This measure does not have atoms (see \cite{Szni98a}, p.~58, or \cite{PortSton78}, p.~197). It is supported by the boundary of $K$ (actually, the boundary of the unbounded component of $K^c$, see \cite{PortSton78}, p.~58). Its total mass is the capacity ${\rm cap}(K)$ of $K$. The equilibrium measure $e_K$ is related to the time of the last visit of $K$ via (see Theorem 3.4, p.~60 of \cite{Szni98a}):
\begin{equation}\label{2.2}
P_y[L_K > 0, L_K \in dt, X_{L_K} \in dz] = p_t(y,z) \,e_K(dz) \,dt,
\end{equation}

\n
where for $w \in W_+$ we write $L_K(w) = \sup\{t > 0; w(t) \in K\}$ (with the convention $\sup \phi = 0$), for the time of last visit of $w$ to $K$.

\medskip
For $B$ a closed ball (see above (\ref{2.1}) for the terminology), we introduce the following (finite) measure on $W^0_B = \{w \in W; H_B(w) = 0\}$ ($\subseteq W_B$), the subset of $W$ of bilateral trajectories that enter $B$ at time $0$ for the first time,
\begin{equation}\label{2.3}
Q_B[(X_{-t})_{t \ge 0} \in A', \, X_0 \in dy, \,(X_t)_{t \ge 0} \in A] = e_B(dy) \,P_y^B[A'] \, P_y[A], \;\mbox{for} \; A,A' \in \cW_+.
\end{equation}

\n
The next lemma states an important compatibility property of the above collection of finite measures. In essence, the measures $Q_B$ should be thought of as the expressions in the (natural) ``local charts'' of $W^*$ with respective domains $W^0_B$ ($\subseteq W_B$) of the intensity measure we are trying to construct (see also below (1.15) of \cite{Szni10a}).
\begin{lemma}\label{lem2.1}
Assume that $B,B'$ are closed balls, and $B$ lies in the interior of $B'$, then 
\begin{equation}\label{2.4}
\theta_{H_B} \circ (1\{H_B < \infty\} \,Q_{B'}) = Q_B.
\end{equation}
\end{lemma}

We postpone the proof of Lemma \ref{lem2.1} for the time being, and explain how the construction of the intensity measure of our basic Poisson point process proceeds. We first observe that for $K$ compact subset of $\IR^d$ we can unambiguously define
\begin{equation}\label{2.5}
Q_K = \theta_{H_K} \circ (1\{H_K < \infty\}\,Q_B), \;\mbox{for any closed ball $B \supseteq K$},
\end{equation}

\n
and this definition coincides with (\ref{2.3}) when $K$ is a closed ball. Indeed, when the closed balls $B_1$ and $B_2$ contain $K$, we can find a closed ball $B'$ containing $B_1 \cup B_2$ in its interior, so that by Lemma \ref{lem2.1} we have
\begin{equation}\label{2.6}
\begin{split}
\theta_{H_K} \circ \big(1\{H_K < \infty\} \,Q_{B_1}\big) & = \theta_{H_K} \big(1 \{H_K < \infty\} \, \theta_{H_{B_1}} \circ (1\{H_{B_1} < \infty\}\,Q_{B'})\big)
\\
& = \theta_{H_K} \circ \big(\theta_{H_{B_1}} (1\{H_K < \infty\}\,Q_{B'})\big) 
\\
&= \theta_{H_K} \circ \big(1\{H_K < \infty\}\,Q_{B'}\big), 
\end{split}
\end{equation}
and a similar identity holds with $B_2$ in place of $B_1$. Hence (\ref{2.5}) is unambiguous and agrees with (\ref{2.3}) when $K$ is a closed ball.

\medskip
We now come to the theorem constructing what will be in essence up to a non-negative multiplicative factor the intensity measure of the Poisson point process defining the Brownian interlacements.

\begin{theorem}\label{theo2.2}
There exists a unique $\sigma$-finite measure $\nu$ on $(W^*, \cW^*)$ such that for each compact subset $K$ of $\IR^d$
\begin{equation}\label{2.7}
1_{W^*_K} \,\nu = \pi^* \circ Q_K.
\end{equation}
\end{theorem}

\begin{proof}
Given a sequence of compact subsets $K_n \uparrow \IR^d$, we have $W^* = \bigcup_{n \ge 0} W^*_{K_n}$, and the uniqueness of $\nu$ is immediate. For the existence of $\nu$, it suffices to check that for $K \subseteq K'$ compact subsets of $\IR^d$
\begin{equation}\label{2.8}
\pi^* \circ Q_K = \pi^* \circ (1_{W_K} \,Q_{K'}) .
\end{equation}

\n
Indeed, one then defines $\nu$ so that $1_{W^*_{K_n}} \nu = \pi^* \circ Q_{K_n}$ for some $K_n \uparrow \IR^d$, noting that the restriction to $W^*_{K_n}$ of $\pi^* \circ Q_{K_{n+1}}$ equals $\pi^* \circ(1_{W_{K_n}}Q_{K_{n+1}})$ which equals $\pi^* \circ Q_{K_n}$ by (\ref{2.8}). Using (\ref{2.8}) once again, one sees that $\nu$ does not depend on the sequence $K_n$, and (\ref{2.7}) immediately follows (choosing $K_0 = K$). To prove (\ref{2.8}), we note that when the closed ball $B$ contains $K'$, then by (\ref{2.5}) we have
\begin{equation*}
Q_K = \theta_{H_K} \circ (1_{W_K}  Q_B) \; \mbox{and} \; 1_{W_K}  Q_{K'} = 1_{W_K}  \theta_{H_{K'}} \circ (1_{W_{K'}}  Q_B) = \theta_{H_{K'}} \circ (1_{W_K}  Q_B).
\end{equation*}

\n
Hence, the images under $\pi^*$ of $Q_K$ and $1_{W'_K}  Q_{K'}$ coincide, that is (\ref{2.7}) holds. 
\end{proof}

We are now reduced to the 

\medskip\n
{\it Proof of Lemma \ref{lem2.1}}:

\medskip\n
It suffices to show that for any continuous compactly supported function $f$: $\IR \r \IR^d$,
\begin{equation}\label{2.9}
\begin{array}{l}
\dis\int e_{B'}(dy') \,E_{y'}\big[H_B < \infty, e^{i \int_0^{H_B} f(v-H_B) \cdot X_v dv} E_{X_{H_B}} \big[e^{i \int^\infty_0 f(v) \cdot X_v dv}\big]\,F'(y',H_B)\big] =
\\[1ex]
\dis\int e_B(dy) \,E^B_y \big[e^{i \int^\infty_0 f(-v) \cdot X_v dv}\big] \,E_y\big[e^{i \int^\infty_0 f(v) \cdot X_v dv}\big], \;\; \mbox{where}
\\[2ex]
F'(y',t) = E^{B'}_{y'} \big[e^{i \int^\infty_0 f(-v-t) \cdot X_v dv}\big], \;\;\mbox{for $y' \in \partial B'$, $t \ge 0$}.
\end{array}
\end{equation}

\n
Indeed the second line of (\ref{2.9}) coincides with $E^{Q_B}[e^{i \int_\IR f(v) \cdot X_v dv}]$, whereas by the strong Markov property at time $H_B$, the first line of (\ref{2.9}) equals
\begin{equation*}
E^{Q_{B'}}[H_B < \infty, e^{i \int_{\IR} f(v- H_B) \cdot X_v dv}] = E^{\theta_{H_B} \circ (1\{H_B<\infty\}Q_{B'})} [e^{i \int_\IR f(v) \cdot X_v dv}],
\end{equation*}

\n
and the claim (\ref{2.4}) follows. We will thus prove (\ref{2.9}). We denote by $P^t_{y,y'}$ the Brownian bridge measure in time $t > 0$, from $y$ to $y'$ in $\IR^d$, see \cite{Szni98a}, p.~137-140, and write $E^t_{y,y'}$ for the corresponding expectation. We consider $y \in B'$ and use the last exit decomposition at the last visit of $B'$, see Theorem 2.12 of \cite{Geto79} and (3.39) on p.~60 of \cite{Szni98a}, to assert that under $P_y$, conditionally on $L_{B'} = t$ (with $t > 0$), and $X_{L_{B'}} = y'$ (with $y' \in \partial B)$, the law of $(X_{L_{B'} + v})_{v\ge 0}$ and $(X_v)_{0 \le v \le t}$ are independent, respectively distributed as $P^B_{y'}$ and $P^t_{y,y'}$ and $(L_{B'}$, $X_{L_{B'}})$ has distribution $p_t(y,y') \,e_{B'}(dy')\, 1\{t > 0\}\,dt$. Writing $U = B^c$ we find as a result that for $y$ in $B'$
\begin{equation}\label{2.10}
\begin{array}{l}
E_y\big[e^{i \int^\infty_0 f(-v) \cdot X_v dv}, H_B = \infty\big] =
\\[2ex]
\dis\int^\infty_0 dt \dis\int e_{B'}(dy') \,E_{y'}^{B'} \big[e^{i \int^\infty_0 f(-v-t) \cdot X_v dv}\big] \,E^t_{y,y'} \big[e^{i \int^t_0 f(-v) \cdot X_v dv}, T_U > t\big] \,p_t(y,y') = 
\\[2ex]
\dis\int^\infty_0 dt \dis\int e_{B'}(dy') \,F'(y',t) \, E^t_{y',y} \big[e^{\int^t_0 f(v-t) \cdot X_v dv}, T_U > t\big] \,p_t(y',y),
\end{array}
\end{equation}

\n
where, in the last step, we used that the law of $(X_{t- \,\cdot})_{0 \le \cdot \le t}$ under $P^t_{y,y'}$ equals $P^t_{y',y}$, cf.~(A.12), p.~139 of \cite{Szni98a}, that $p_t(\cdot,\cdot)$ is symmetric, and the notation from below (\ref{2.9}).

\medskip
We introduce the notation $F(z) = \psi(z) \,E_z[e^{i \int^\infty_0 f(v) \cdot X_v dv}]$ for $z \in \IR^d$, where $\psi$ is a continuous compactly supported $[0,1]$-valued function, which equals to $1$ on a neighborhood of $B$. We pick $\varepsilon > 0$ and introduce the finite measure
\begin{equation}\label{2.11}
\begin{split}
m_\varepsilon(dy) & = \mbox{\f $\dis\frac{1}{\varepsilon}$} \; \dis\int dz \,F(z) \,P_z[H_B < \ve, X_\varepsilon \in dy] 
\\ 
 &= \mbox{\f $\dis\frac{1}{\varepsilon}$}  \; \dis\int dz \,F(z) \,P^\varepsilon_{z,y} [H_B < \varepsilon] \,p_\varepsilon (z,y) \,dy.
\end{split}
\end{equation}

\n
We integrate (\ref{2.10}) with respect to $m_\varepsilon(dy)$. The first line of (\ref{2.10}) yields
\begin{align*}
& \dis\int dz\,F(z) \,\mbox{\f $\dis\frac{1}{\varepsilon}$} \;E_z\big[H_B < \varepsilon, \;E_{X_\varepsilon} [H_B = \infty, e^{i \int^\infty_0 f(-v) \cdot X_v dv}]\big] \stackrel{{\rm definition} \;P_\point^B}{=}
\\ 
& \dis\int dz\,F(z) \,\mbox{\f $\dis\frac{1}{\varepsilon}$} \;E_z\big[H_B < \varepsilon, \;P_{X_\varepsilon} [H_B = \infty]\,E^B_{X_\varepsilon} [e^{i \int^\infty_0 f(-v) \cdot X_v dv}]\big] \underset{\rm property}{\stackrel{\rm simple \; Markov}{=}}
\\ 
& \dis\int dz\,F(z) \,\mbox{\f $\dis\frac{1}{\varepsilon}$} \;E_z\big[H_B < \varepsilon, H_B \circ \theta_\ve = \infty, \, E^B_{X_\ve} [e^{i \int^\infty_0 f(-v) \cdot X_v dv}]\big] = 
\\ 
&  \dis\int dz\,F(z) \,\mbox{\f $\dis\frac{1}{\varepsilon}$} \;E_z\big[0 < L_B < \ve, \,E^B_{X_\ve} [e^{i \int^\infty_0 f(-v) \cdot X_v dv}]\big] .
\end{align*}

\n
By the last exit decomposition of Brownian motion starting at $z$ at the last visit of $B$, the above expression equals
\begin{align*}
& \dis\int dz\,F(z) \,\mbox{\f $\dis\frac{1}{\varepsilon}$} \;\dis\int^\ve_0 ds \,p_s(z,y) \,e_B(dy) \,E^B_y \big[E^B_{X_{\ve -s}} [e^{i \int^\infty_0 f(-v) \cdot X_v dv}]\big] =
\\
&\dis\int e_B(dy) \, \mbox{\f $\dis\frac{1}{\varepsilon}$} \;\dis\int^\ve_0 ds \Big(\int p_s(y,z) \,F(z) \,dz\Big)\,E^B_y \big[E_{X_{\ve -s}}^B [e^{i \int^\infty_0 f(-v) \cdot X_v dv}]\big]  \;\mbox{(by symmetry of $p_s(\cdot,\cdot)$)}.
\end{align*}

\n
Using the fact that for $y\in \partial B$, $P^B_y$-a.s., as $u \r 0$, $P^B_{X_u}$ converges weakly to $P^B_y$ (one could actually also invoke the Markov property under $P^B_y$ at this point), and dominated convergence, as $\ve \r 0$, the above quantity tends to
\begin{equation}\label{2.12}
\dis\int e_B(dy) \,E_y\big[e^{i \int^\infty_0 f(v) \cdot X_v dv}\big] \,E^B_y \big[e^{i \int^\infty_0 f(-v) \cdot X_v dv}\big].
\end{equation}

\n
Using the symmetry of $p_\ve(\cdot,\cdot)$ and the fact that for $y \in U = B^c$, $P^\ve_{z,y}[H_B < \ve] = P^\ve_{y,z} [H_B < \ve]$ (when $z \in B$, note that $P_z$-a.s., $\wt{H}_B = 0$, and both terms equal $1$), the last line of (\ref{2.10}) after integration with respect to $m_\ve(dy)$ yields
\begin{align*}
& \dis\int e_{B'} (dy') \dis\int^\infty_0 \mbox{\f $\dis\frac{dt}{\varepsilon}$} \;F'(y',t)  \dis\int dz \,dy \,F(z) 
\\
&\qquad  E^t_{y',y} \big[e^{i \int^t_0 f(v-t) \cdot X_v dv}, T_U > t\big]\, p_t(y',y) \,P^\ve_{y,z}[H_B < \ve] \,p_\ve(y,z) =
\\[1ex]
& \dis\int e_{B'} (dy') \dis\int^\infty_0 \mbox{\f $\dis\frac{dt}{\varepsilon}$} \; F'(y',t) \,E_{y'}  \big[e^{i \int^t_0 f(v-t) \cdot X_v dv}, t < H_B < t + \ve, \,F(X_{t+\ve})\big] =
\\ 
& \dis\int e_{B'} (dy')\,E_{y'} \Big[H_B < \infty, \mbox{\f $\dis\frac{1}{\varepsilon}$}  \;\dis\int^{H_B}_{(H_B - \ve)_+} e^{i \int^t_0 f(v-t) \cdot X_v dv} F(X_{t + \ve}) \,F'(y',t)\,dt\Big].
\end{align*}

\n
Applying dominated convergence as $\ve \r 0$, the above quantity tends to 
\begin{equation}\label{2.13}
\dis\int e_{B'}(dy') \,E_{y'} \big[H_B < \infty, e^{i \int^{H_B}_0 f(v-H_B) \cdot X_v dv} E_{X_{H_B}} [e^{i \int^\infty_0 f(v) \cdot X_v dv}]\,F'(y', H_B)\big].
\end{equation}

\n
Comparing (\ref{2.12}) and (\ref{2.13}), we have shown (\ref{2.9}). This proves Lemma \ref{lem2.1}. \hfill $\square$

\begin{remark}\label{rem2.3} \rm  ~

\medskip\n
1) When $B$ is a closed ball, for $w \in W_B$ we can define $L_B(w) = \sup\{t \in \IR, L_B(w) \in B\}$. Applying the last exit decomposition formula, as in (\ref{2.10}), to the forward term in (\ref{2.3}), we see that for $A, A' \in W_+$ we have
\begin{equation}\label{2.14}
\begin{array}{l}
Q_B [(X_{-t})_{t \ge 0} \in A', \;(X_s)_{0 \le s \le L_B} \in \cdot, \,(X_{L_B+s})_{s \ge 0} \in A] =
\\[1ex]
\dis\int^\infty_0 dt \dis\int e_B (dy') \,e_B (dy) \,P^B_{y'}[A'] \,P^t_{y',y}[\cdot ] \,P^B_y[A]\,p_t(y',y)
\end{array}
\end{equation}

\n
(where $(X_s)_{0  \le s \le L_B}$ is viewed as a random element of the space $\cT$ of continuous $\IR^d$-valued trajectories of positive finite duration, measurably identified with $C([0,1], \IR^d) \times (0,\infty)$ via the map $(w,T) \r w(\frac{\cdot}{T}) \in \cT$).

\medskip
From this identity using the symmetry of $p_t(\cdot,\cdot)$ and the already used fact that $P^t_{y',y}$ is the image of $P^t_{y,y'}$ under time reversal, we find (analogously to (1.40) of \cite{Szni10a}) that
\begin{equation}\label{2.15}
\mbox{$(X_{L_B - t})_{t \in \IR}$ under $Q_B$ has law $Q_B$.}
\end{equation}

\n
One can introduce on $W^*$ the time inversion $w^* \r \overset{\vee}{w}\,{\!^*}$, where $\overset{\vee}{w}\,{\!^*} = \pi^*(\overset{\vee}{w})$, with $\overset{\vee}{w}(t) = w(-t)$ for $t \in \IR$, and $w \in W$ arbitrary such that $\pi^*(w) = w^*$. It follows from (\ref{2.15}) that the image $(\pi^* \circ Q_B)^\vee$ of $(\pi^* \circ Q_B)$ under time inversion coincides with $\pi^* \circ Q_B$. By Theorem \ref{theo2.2} it follows that the image $\overset{\vee}{\nu}$ of $\nu$ under time inversion satisfies
\begin{equation}\label{2.16}
\overset{\vee}{\nu} = \nu .
\end{equation}

\medskip\n
2) In an easier fashion one sees that when one considers $y \in \IR^d$ and denotes by $\tau_y$ the translation by $-y$ on $W^*$, which to $w^*$ associates $w^* - y$ (defined in an obvious way), or a linear isometry $R$ of $\IR^d$, and its natural operation $w^* \r R w^*$ on $W^*$, we find that
\begin{align}
&\mbox{$\nu$ is invariant under $\tau_y$}, \label{2.17}
\\[1ex]
& \mbox{$\nu$ is invariant under $R$}, \label{2.18}
\end{align}

\medskip\n
3) When $\lambda >0$, one can define the scaling $s_\lambda(w) = \lambda w(\frac{\cdot}{\lambda^2})$ for $w \in W$, and on $W^*$ via $s_\lambda(w^*) = \pi^* (s_\lambda(w))$, for $w \in W$ arbitrary such that $\pi^*(w) = w^*$. From the identity valid for arbitrary $\lambda >0$ and closed ball $B$
\begin{equation}\label{2.19}
s_\lambda \circ Q_B = \lambda^{2-d} Q_{\lambda B},
\end{equation}
one finds by Theorem \ref{theo2.2} that
\begin{equation}\label{2.20}
s_\lambda \circ  \nu = \lambda^{2-d} \nu, \; \mbox{for $\lambda > 0$}.
\end{equation} 

\medskip\n
4) One knows that when the compact set $K (\subseteq \IR^d)$ lies in the interior of the closed ball $B$, then $e_K(\cdot) = P_{e_B} [H_K < \infty, X_{H_K} \in \cdot]$, see Theorem 1.10, p.~58 of \cite{PortSton78}. By (\ref{2.3}) and (\ref{2.5}) we thus see that
\begin{equation}\label{2.21}
Q_K [(X_t)_{t \ge 0} \in \cdot] = P_{e_K}.
\end{equation} 

\n
We can now introduce the measurable map from $W^*_K$ into $W_+$ defined by $w^* \in W^*_K \r w^{*,K,+} = \big(w(H_K + t)\big)_{t \ge 0}$, for any $w \in W_K$ with $\pi^*(w) = w^*$. Then, by (\ref{2.7}) and (\ref{2.21}) we see that
\begin{equation}\label{2.22}
\mbox{the image of $1_{W^*_K} \nu$ under $w^* \r w^{*,K,+}$ equals $P_{e_K}$}.
\end{equation} 
\hfill $\square$
\end{remark}

As a next step we introduce the canonical space for the Brownian interlacement point process, namely the space of point measures on $W^* \times \IR_+$, 
\begin{align}
\Omega = \big\{ & \o = \dsl_{i \ge 0} \delta_{(w^*_i, \alpha_i)}, \;\mbox{with} \;(w^*_i, \alpha_i) \in W^* \times [0,\infty) \;\mbox{and} \;\o(W^*_K \times [0,\alpha]) < \infty,\label{2.23}
\\[-1ex]
& \mbox{for any compact subset $K$ of $\IR^d$ and $\alpha \ge 0\big\}$}. \nonumber
\end{align}  

\n
We endow $\Omega$ with the $\sigma$-algebra $\cA$ generated by the evaluation maps $w \r \o(B)$, for $B \in \cW^* \otimes \cB(\IR_+)$, and denote by $\IP$ the law on $(\Omega, \cA)$ of the Poisson point measure with intensity measure $\nu \otimes d \alpha$ on $W^* \times \IR_+$. For $K$ compact subset of $\IR^d$ and $\alpha \ge 0$, we consider the measurable function on $\Omega$ with value in the set of finite point measures on $W_+$
\begin{align}
& \mu_{K,\alpha}(\o) = \dsl_{i \ge 0} 1\{\alpha_i \le \alpha, w_i^* \in W^*_K\} \, \delta_{w_i^{*, K,+}}, \;\mbox{when} \;\; \o = \dsl_{i \ge 0} \delta_{(w_i^*,\alpha_i)} \in \Omega, \label{2.24}
\intertext{with the notation from above (\ref{2.22}), We readily see by (\ref{2.22}) that}
& \mbox{$\mu_{K,\alpha}$ is a Poisson point process on $W_+$ with intensity measure $\alpha P_{e_K}$}. \label{2.25}
\end{align}

\medskip\n
As a direct consequence of (\ref{2.16}) - (\ref{2.20}) we obtain the following invariance properties.
\begin{proposition}\label{prop2.4}
$\IP$ is invariant under the following transformations on $\Omega$:
\begin{align}
\o & = \dsl_{i \ge 0} \delta_{(w^*_i,\alpha_i)} \longrightarrow \dsl_{i \ge 0} \delta_{(\overset{\vee}{w}_i\,{\!\!\!^*},\alpha_i)}, \label{2.26}
\\[2ex]
\o & = \dsl_{i \ge 0} \delta_{(w^*_i,\alpha_i)} \longrightarrow \dsl_{i \ge 0} \delta_{(w^*_i - y,\alpha_i)} \label{2.27} \quad \qquad  \mbox{(with $y \in \IR^d$)},
\\[2ex]
\o & = \dsl_{i \ge 0} \delta_{(w^*_i,\alpha_i)} \longrightarrow \dsl_{i \ge 0} \delta_{(Rw_i^*,\alpha_i)} \qquad \quad \; \mbox{(with $R$ linear isometry of $\IR^d$)}, \label{2.28}
\\[2ex]
\o & = \dsl_{i \ge 0} \delta_{(w^*_i,\alpha_i)} \longrightarrow \dsl_{i \ge 0} \delta_{(s_\lambda(w_i^*),\lambda^{2-d} \alpha_i)} \quad \mbox{(with $\lambda > 0$)}. \label{2.29}
\end{align}
\end{proposition}

\n
When $\o \in \Omega$, $\alpha \ge 0$, $r \ge 0$, the formula (see the beginning of Section 1 for notation)
\begin{equation}\label{2.30}
\begin{array}{l}
\mbox{$\cI^\alpha_r(\o) = \bigcup\limits_{i \ge 0: \alpha_i \le \alpha} \;\;  \bigcup\limits_{s \in \IR} B(w_i(s),r)$, where $\o \in \dsl_{i \ge 0} \delta_{(w_i^*,\alpha_i)}$}
\\[2ex]
\mbox{and $\pi^*(w_i) = w_i^*$ for $i \ge 0$},
\end{array}
\end{equation}

\n
defines the {\it Brownian interlacement at level $\alpha$ with radius $r$}. By construction, see (\ref{2.23}), this is a closed subset of $\IR^d$. When $r = 0$, we refer to $\cI_0^\alpha(\o)$ as the {\it Brownian fabric at level $\alpha$}. The terminology is truly pertinent when $d = 3$, see (\ref{2.36}) below. The complement of $\cI^\alpha_r(\o)$ is an open subset of $\IR^d$, the {\it vacant set at level $\alpha$ and radius $r$}:
\begin{equation}\label{2.31}
\cV^\alpha_r(\o) = \IR^d \backslash \cI_r^\alpha(\o), \;\mbox{for} \; \o \in \Omega, \,\alpha \ge 0, \,r \ge 0.
\end{equation}

\n
The set $\Sigma$ of closed (possibly empty) subsets of $\IR^d$ can be endowed with the $\sigma$-algebra $\cF$ generated by the sets $\{F \in \Sigma; \,F \cap K = \emptyset\}$, where $K$ varies over the compact subsets of $\IR^d$, see \cite{Math75}, p.~27 (incidentally, this is the Borel $\sigma$-algebra for a metrizeable topology for which $\Sigma$ is compact, see \cite{Math75}, p.~3 and 27). The $\Sigma$-valued map $\cI^\alpha_r$ on $\Omega$ is measurable, and one can thus consider its law $Q^\alpha_r$ on $(\Sigma, \cF)$.

\begin{proposition}\label{prop2.5} ($\alpha \ge 0$, $r \ge 0$, $y \in \IR^d$, $R$ linear isometry, $\lambda > 0$)

\medskip
$Q^\alpha_r$ is determined by the identity
\begin{equation}\label{2.32}
Q^\alpha_r(\{F \in \Sigma; \, F \cap K = \emptyset\}) = \IP[\cI^\alpha_r \cap K = \emptyset] = e^{-\alpha\, {\rm cap}(K_r)}, \; \mbox{for $K \subseteq \IR^d$ compact},
\end{equation}

\n
where $K_r = K + B(0,r)$ the set of points within $|\cdot |$-distance at most $r$ from $K$. Moreover, under $\IP$,
\begin{align}
&\mbox{$\cI^\alpha_r + y$ has same law as $\cI^\alpha_r$} \qquad  \,\mbox{(translation invariance)}, \label{2.33}
\\[1ex]
&\mbox{$R(\cI^\alpha_r)$ has same law as $\cI^\alpha_r$} \qquad  \;\mbox{(isotropy)}, \label{2.34}
\\[1ex]
&\mbox{$\lambda \cI^\alpha_r$ has same law as $\cI^{\lambda^{2-d}\alpha}_{\lambda r}$}  \qquad \mbox{(scaling)}, \label{2.35}
\\[1ex]
&\mbox{$\cI^\alpha_0$ is a.s.-connected, when $d=3$, and disconnected, when $d \ge 4$ and $\alpha > 0$.} \label{2.36}
\end{align}
\end{proposition}

\begin{proof}
Note that $\{\cI^\alpha_r \cap K = \emptyset\} = \{\o \in \Omega; \,\o(W^*_{K_r} \times [0,\alpha]) = 0\}$, so that
\begin{equation*}
\IP[\cI^\alpha_r \cap K = \emptyset ]  = c^{-\alpha \nu(W^*_{K_r})} \stackrel{(\ref{2.7}), (\ref{2.21})}{=} e^{-\alpha \,{\rm cap}(K_r)}, \; \mbox{whence (\ref{2.32})}.
\end{equation*}

\n
The identity (\ref{2.32}) determines $Q^\alpha_r$ on a $\pi$-system generating $\cF$, and the first claim follows. The identities (\ref{2.33}) - (\ref{2.35}) are direct consequences of (\ref{2.27}) - (\ref{2.29}) (alternatively, they can be derived from (\ref{2.32})). The last claim is a direct consequence of the fact that for any closed ball $B, Q_B(dw) \otimes Q_B(dw')$-a.s., $w(\IR) \cap w'(\IR) \not= \emptyset$, when $d=3$, but $w(\IR) \cap w'(\IR) = \emptyset$, when $d \ge 4$, as a consequence of the fact that Brownian paths meet each other when $d=3$, and, when starting from different points, miss each other when $d \ge 4$, see \cite{DvorErdoKaku50} (when $d \ge 4$ we also use (\ref{2.15}) to take care of the bilateral nature of the paths).
\end{proof}

As a last topic of this section, we introduce the occupation-time measure of Brownian interlacements at level $\alpha \ge 0$ and discuss some of its properties. It is the locally finite (or Radon) measure on $\IR^d$ defined for $\o \in \Omega$, $\alpha \ge 0$, $A \in \cB(\IR^d)$ by
\begin{equation}\label{2.37}
\begin{split}
\cL_\alpha(\o) (A) & = \dsl_{i \ge 0: \alpha_i \le \alpha} \; \dis\int_{\IR} 1\{w_i(s) \in A\} \,ds, \;\;\mbox{where $\o = \dsl_{i \ge 0} \delta_{(w^*_i,\alpha_i)}$ with}
\\[-2ex]
& \qquad \qquad \qquad \quad  \mbox{$\pi^*(w_i) = w_i^*$, for $i \ge 0$},
\\
& = \big\langle \o, f_A \otimes 1_{[0,\alpha]} \big\rangle,
\end{split}
\end{equation}

\n
where $f_A(w^*) = \int_\IR 1_A\big(w(s)\big)ds$, for $w \in W$ arbitrary such that $\pi^*(w) = w^*$, and $\langle \omega,h\rangle$ stands for the integral of $h$ with respect to $\omega$. Note that by the choice of $\Omega$ and $W^*$, the expression in (\ref{2.37}) is finite when $A$ is a bounded set. From the second line of (\ref{2.37}), the dependence in $\o$ is $\cA$-measurable, and $\cL_\alpha$ defines a random measure on $\IR^d$, see Chapter 1 of \cite{Kall76}. Further when $A \in B(\IR^d)$ is bounded and $B$ a closed ball containing $A$, then for $\alpha \ge 0$
\begin{equation}\label{2.38}
\begin{split}
\IE[\cL_\alpha(\o)(A)] & = \IE\big[\big\langle \o, f_A \otimes 1_{[0,\alpha]}\big\rangle\big] = \alpha \big\langle \nu,f_A\big\rangle
\\
&\!\!\!\!\!\! \stackrel{(\ref{2.7})_,(\ref{2.3})}{=} \alpha \,E_{e_B} \Big[\dis\int^\infty_0 1_A(X_s) \,ds\Big] = \alpha \dis\int e_B(dy) \,G(y,y') \,1_A(y')\,dy'
\\
&=  \alpha \,|A| \;\; \mbox{(with $|A|$ the Lebesgue measure of $A$)},
\end{split}
\end{equation}

\medskip\n
since $\int e_B(dy)\,G(y,y') = 1$ for $y' \in B$, by (\ref{2.2}), (\ref{2.1}). In other words:
\begin{equation}\label{2.39}
\mbox{the intensity measure of $\cL_\alpha$ equals $\alpha \,dy$}.
\end{equation}
Observe also that by inspection of (\ref{2.37})
\begin{equation}\label{2.40}
\mbox{the support of $\cL_\alpha$ coincides with $\cI^\alpha_0$}.
\end{equation}

\n
The next result will be very useful and provides an expression for the Laplace transform of $\cL_\alpha$.
\begin{proposition}\label{prop2.6}
When $V$ is a bounded, measurable, compactly supported function on $\IR^d$, and\begin{equation}\label{2.41}
\mbox{$\|G\,|V|\,\|_{L^\infty(\IR^d)} < 1$ (see (\ref{1.10}) for notation)},
\end{equation}

\n
then $I - GV$ is a bounded invertible operator on $L^\infty(\IR^d)$ and for any $\alpha \ge 0$,
\begin{equation}\label{2.42}
\IE\big[\exp\big\{\big\langle \cL_\alpha, V\big\rangle\big\}\big] = \exp\big\{\alpha \big\langle V, (I-GV)^{-1} 1\big\rangle\big\},
\end{equation}

\medskip\n
with the notation (\ref{1.9}) and $\langle \cL_\alpha, V\rangle$ the integral of $V$ with respect to $\cL_\alpha$.
\end{proposition}

\begin{proof}
Up to routine modifications, this follows by similar arguments as Theorem 2.1 of \cite{Szni11c}.
\end{proof}

\begin{remark}\label{rem2.7} \rm 
As a direct application of the invariance property (\ref{2.29}) and the definition (\ref{2.9}) we see that for $\alpha \ge 0$, the occupation-time measure $\cL_\alpha$ has the following scaling invariance
\begin{equation}\label{2.43}
\cL_\alpha \stackrel{\rm law}{=} \lambda^2 \,h_\lambda \circ \cL_{\lambda^{d-2}\alpha}, \;\; \mbox{for $\lambda > 0$},
\end{equation}

\n
where $h_{\lambda}(y) = \lambda y$ denotes the homothety of ratio $\lambda$ on $\IR^d$. One can also recover the identity (\ref{2.43}) from (\ref{2.42}). \hfill $\square$
\end{remark}

\section{Scaling limits of occupation times}
\setcounter{equation}{0}

In this section we study the scaling limit of the random field of occupation times of random interlacements on $\IZ^d$, $d \ge 3$. In the constant intensity regime (\ref{0.2}), we show that $\cL^N$ (see (\ref{0.1})) converges in distribution to the occupation-time measure of Brownian interlacements, see Theorem \ref{theo3.2}. In the high intensity regime (\ref{0.4}), we show that $\wh{\cL}^N$ (see (\ref{0.5})) converges in distribution to the massless Gaussian free field, see Theorem \ref{theo3.3} and Corollary \ref{cor3.5}.

\medskip
We tacitly endow the set of Radon measures on $\IR^d$ with the topology of the vague convergence (see A.7 of \cite{Kall76}). Given a positive sequence $(u_N)_{N \ge 1}$, it follows from (\ref{1.12}) that 
\begin{equation}\label{3.1}
\ov{\IE} [\cL^N ([0,1)^d)] = \mbox{\f $\dis\frac{1}{d}$} \;N^{d-2} u_N .
\end{equation}

\n
This quantity equals $\alpha (> 0)$ in the constant intensity regime (\ref{0.2}), and tends to $\infty$ in the high intensity regime (\ref{0.4}). The low intensity regime (when $N^{d-2} u_N \r 0$) leads to a convergence in distribution of $\cL^N$ to the null measure on $\IR^d$, and we will only focus on the constant and high intensity regimes.

\medskip
We begin with some preparation. We consider an integer $M \ge 1$, and
\begin{equation}\label{3.2}
\mbox{a continuous function $V$ on $\IR^d$ with support in $C_M = [-M,M]^d$}.
\end{equation}

\medskip\n
We recall the notation (\ref{1.8}), (\ref{1.10}) from Section 1. With a slight abuse of notation, we still denote by $V$ the restriction of $V$ to $\IL_N$ (see (\ref{1.4})).

\begin{proposition}\label{prop3.1} (under (\ref{3.2}))
\begin{align}
& \sup\limits_{N \ge 1} \, \|G_N \,|V|\,\|_{L^\infty(\IL_N)}  \le c_0(M) \,\|V\|_{L^\infty(C_M \cap \IL_N)} \le c_0(M) \,\|V\|_{L^\infty(\IR^d)}. \label{3.3}
\\[1ex]
&\lim\limits_N \big\langle V, (G_N V)^{n-1} 1\big\rangle_{\IL_N} = \big\langle V,(GV)^{n-1} 1\big\rangle, \;\; \mbox{for all $n \ge 1$}. \label{3.4}
\end{align}
\end{proposition}

\begin{proof}
We begin with the proof of (\ref{3.3}), and assume, without loss of generality, that $\|V\|_{L^\infty(C_M \cap \IL_N)} = 1$. By a classical harmonicity (or martingale) argument one sees that $\sup_{y \in \IL_N}$ $|G_N|V|(y)| = \sup_{y \in C_M \cap \IL_N} |G_N |V|(y)|$, since $V$ is supported in $C_M$. It then follows from (\ref{1.2}), (\ref{1.8}) that
\begin{equation}\label{3.5}
\begin{array}{r}
\sup\limits_{y \in C_M \cap \IL_N} |G_N |V|(y)| \le \mbox{\f $\dis\frac{c}{N^2}$} \;\Big(1 + \dsl_{0 < |x|_\infty \le 2 MN} \;\; \mbox{\f $\dis\frac{1}{|x|^{d-2}}$}\Big) \le c' M^2
\\
\\[-2ex]
\mbox{($x$ belongs to $\IZ^d$ in the sum)},
\end{array}
\end{equation}

\n
and the claim (\ref{3.3}) follows. Note that the continuity of $V$ was not needed for the proof of (\ref{3.3}). We then turn to the proof of (\ref{3.4}). As we now explain, it suffices to show that  for all bounded continuous functions $W$ on $\IR^d$,
\begin{equation}\label{3.6}
\lim\limits_N \;\sup\limits_{y \in C_M \cap \IL_N} |(G_N VW) (y) - (GVW)(y)| = 0.
\end{equation}

\n
Indeed, once (\ref{3.6}) is established, we note that $GVW$ is bounded continuous on $\IR^d$ as a convolution of the locally integrable function $G(\cdot)$ (see (\ref{1.3})) with the compactly supported continuous function $VW$. Then, we observe that for $n \ge 1$,
\begin{align*}
&\sup\limits_{y \in C_M \cap \IL_N} |(G_NV)^n 1(y) - (GV)^n 1(y)| \le \sup\limits_{y \in C_M \cap \IL_N} \big|(G_N V)\big((G_NV)^{n-1} 1- (GV)^{n-1}1\big) (y)\big| +
\\[1ex]
&\sup\limits_{y \in C_M \cap \IL_N} |(G_N V(GV)^{n-1} 1) (y) - (GV(GV)^{n-1} 1)(y)|.
\end{align*}

\n
Using induction over $n \ge 1$, we see that the last term tends to zero with $N$, by (\ref{3.6}) and the observation below (\ref{3.6}), and the first term after the inequality sign tends to zero with $N$ by the induction hypothesis and the first inequality in (\ref{3.3}). Thus, once (\ref{3.6}) is proved, it follows that
\begin{equation}\label{3.7}
\sup\limits_{C_M \cap \IL_N} |(G_N V)^{n-1} 1(\cdot) - (GV)^{n-1} 1(\cdot) | \underset{N}{\longrightarrow} 0, \;\mbox{for each $n \ge 1$}.
\end{equation}

\n
The claim (\ref{3.4}) now follows by a straightforward Riemann sum approximation.

\medskip
We now prove (\ref{3.6}) and introduce the shorthand $F = VW$. The function $F$ is continuous on $\IR^d$ with support in $C_M$. Then, for $\gamma \in (0,1)$, $N \ge 1$, and $y \in C_M \cap \IL_N$, we have
\begin{equation}\label{3.8}
\begin{array}{l}
|G_N F(y) - GF(y)| = \Big| \mbox{\f $\dis\frac{1}{N^d}$} \;\dsl_{y' \in \IL_N} g_N(y') \,F(y-y') - \dis\int_{\IR^d} G(y) \,F(y-y')\,dy'\Big| \le
\\[1ex]
\Big(\mbox{\f $\dis\frac{1}{N^d}$}  \;\dsl_{|y'|_\infty \le \gamma} g_N(y') + \dis\int_{|y'|_\infty \le \gamma} G(y') \,dy' \Big)\|F\|_{L^\infty(\IR^d)} \,+
\\[1ex]
\Big|\mbox{\f $\dis\frac{1}{N^d}$} \; \dsl_{\gamma \le |y'|_\infty \le 2M} g_N(y') \,F(y-y') - \dis\int_{\gamma \le |y'|_\infty \le 2M} G(y') F(y-y') \,dy'\Big| .
\end{array}
\end{equation}

\n
The functions $G(\cdot) \,F(y- \cdot)$ are uniformly continuous on $\{y' \in \IR^d; \gamma \le |y'|_\infty \le 2 M\}$ uniformly in $y \in C_M$. By (\ref{1.7}) and a Riemann sum approximation argument we see that the last term of (\ref{3.8}) tends to zero uniformly in $y \in C_M \cap \IL_N$. On the other hand, the second line of (\ref{3.8}) is bounded by $(\frac{c}{N^2} (\gamma^2 N^2 + 1) + c' \gamma^2)\,\|F\|_{L^\infty (\IR^d)}$. Taking a limsup in $N$, and then letting $\gamma$ tend to zero, we obtain (\ref{3.6}). The claim (\ref{3.4}) now follows, and this completes the proof of Proposition \ref{prop3.1}.
\end{proof}

We are now ready to state the convergence result for the law of $\cL^N$ in the constant intensity regime (i.e. $N^{d-2} u_N = d \alpha$, see (\ref{0.2})). The definition of $\cL_\alpha$ appears in (\ref{2.37}).

\begin{theorem}\label{theo3.2} (under (\ref{0.2}))
\begin{equation}\label{3.9}
\mbox{$\cL^N$ converges in distribution to $\cL_\alpha$, as $N \r \infty$.}
\end{equation}
\end{theorem}

\begin{proof}
By Theorem 4.2, p.~22 of \cite{Kall76}, it suffices to show that for any continuous compactly supported function $V$ on $\IR^d$
\begin{equation}\label{3.10}
\mbox{$\langle \cL^N,V\rangle$ converges in distribution to $\langle \cL_\alpha, V\rangle$, as $N \r \infty$}.
\end{equation}

\n
Without loss of generality we assume that $V$ satisfies (\ref{3.2}) and $\|V\|_{L^\infty(\IR^d)} \le 1$. By (\ref{1.11}) and (\ref{3.3}) we see that for $|z| < \frac{1}{c_0}$ (with $c_0$ from (\ref{3.3})) and $N \ge 1$,
\begin{equation}\label{3.11}
\begin{split}
\overline{\IE} \big[\exp\{z \big\langle \cL^N, V\big\rangle\big\}\big] & = \exp\big\{\alpha \big\langle zV, (I-zG_N V)^{-1}1\big\rangle_{\IL_N}\big\}
\\[1ex]
& = \exp\big\{\alpha \dsl_{n \ge 1} z^n\big\langle V, (G_N V)^{n-1}1\big\rangle_{\IL_N}\big\}.
\end{split}
\end{equation}
In particular this shows (with (\ref{3.3})) that
\begin{equation}\label{3.12}
\sup\limits_{N \ge 1} \ov{\IE} \big[\cosh \big(r \big\langle \cL^N,V\big\rangle\big)\big] < \infty, \;\; \mbox{when $r < c_0^{-1}$}.
\end{equation}

\n
Thus, the laws of $\langle \cL^N,V\rangle$, $N \ge 1$, are tight, and the random variables $e^{z \langle \cL^N,V\rangle}$, $N \ge 1$, $|Rez| \le r(< c_0^{-1})$, are uniformly integrable. Hence, if along some subsequence $N_k$, $k \ge 1$, the random variables $\langle \cL^N,V\rangle$ converge in distribution to a random variable $U$, it follows from Theorem 5.4, p.~32 of \cite{Bill68}, that for $|z| < c_0^{-1}$ in $\IC$ one has
\begin{equation}\label{3.13}
\begin{split}
E[e^{zU}] & = \lim\limits_k \;\ov{\IE} [e^{z \langle \cL^{N_k},V\rangle}]
\\
& \!\!\! \stackrel{(\ref{3.11})}{=} \lim\limits_k \;\exp \Big\{\alpha \dsl_{n \ge 1} z^n \big\langle V, (G_{N_k} V)^{n-1} 1\big\rangle_{\IL_N}\Big\}
\\
&\!\!\!\!\!\!\! \stackrel{(\ref{3.3}),(\ref{3.4})}{=}  \exp \Big\{\alpha \dsl_{n \ge 1} z^n \big\langle V, (GV)^{n-1} 1\big\rangle\Big\} \stackrel{(\ref{2.42})}{=} \IE [e^{z \langle \cL_\alpha,V\rangle}],
\end{split}
\end{equation}

\n
where in the last step we have used the fact that $\|G\,|V|\,\|_{L^\infty(\IR^d)} \le c_0$ as a result of (\ref{3.3}) and (\ref{3.6}). By analyticity, the first and the last expression in (\ref{3.13}) are equal in the strip $|Rez| < c_0^{-1}$ of the complex plane. Hence $U$ and $\langle \cL_\alpha, V\rangle$ have the same characteristic function, and we have shown (\ref{3.10}). This completes the proof of Theorem \ref{theo3.2}.
\end{proof}

We now turn to the discussion of the high intensity regime (\ref{0.4}) (i.e. $N^{d-2} u_N \r \infty)$, and recall the definition of the random signed measure $\wh{\cL}^N$ from (\ref{0.5}):
\begin{equation}\label{3.14}
\wh{\cL}^N = \sqrt{\mbox{\f $\dis\frac{d}{2N^{d-2} u_N}$}} (\cL^N - \IE[\cL^N]) \stackrel{(\ref{1.12})}{=} \sqrt{\mbox{\f $\dis\frac{d}{2N^{d-2} u_N}$}} \Big(\cL^N - \mbox{\f $\dis\frac{u_N}{dN^2}$} \;\dsl_{y \in \IL_N} \delta_y\Big).
\end{equation}

\n
For a bounded measurable function $V$ on $\IR^d$ with compact support, the notation $\langle \wh{\cL}^N,V\rangle$ will refer to the integral of $V$ with respect to $\wh{\cL}^N$. In essence, the next result states that $\wh{\cL}^N$ tends in distribution to the massless Gaussian free field, see (\ref{1.15}). This fact will be made precise in Corollary \ref{cor3.5} below.

\begin{theorem}\label{theo3.3} (under (\ref{0.4}))

\medskip
When $V$ is continuous compactly supported function on $\IR^d$, then, as $N \r \infty$, 
\begin{align}
&\mbox{$\langle \wh{\cL}^N,V\rangle$ converges in distribution to a centered Gaussian variable} \label{3.15}
\\[-1ex]
&\mbox{with variance $E(V,V) \stackrel{(\ref{1.14})}{=} \dis\int V(y) \,G(y-y') \,V(y') \,dy \, dy'$}. \nonumber
\end{align}
\end{theorem}

\begin{proof}
Without loss of generality, we assume, for convenience, that (\ref{3.2}) holds and that $\|V\|_{L^\infty (\IR^d)} \le 1$. Further, we use the shorthand $a_N = (\frac{2}{d} N^{d-2} u_N)^{\frac{1}{2}}$, and assume that $a_N \ge 1$, for all $N \ge 1$ (this possibly involves the replacement of finitely many $u_N$). We then see that for $|z| < c_0^{-1}$ and $N \ge 1$, 
\begin{equation}\label{3.16}
\begin{split}
\ov{\IE}[e^{z \langle \wh{\cL}^N,V\rangle}] &= \ov{\IE} \big[\exp \Big\{\mbox{\f $\dis\frac{z}{a_N}$} \big(\big\langle \cL^N,V\big\rangle - \mbox{\f $\dis\frac{u_N}{d}$} \;N^{d-2} \big\langle V,1\big\rangle_{\IL_N}\big)\Big\}\Big]
\\[1ex]
&\!\!\!\!\!\!\!\! \stackrel{(\ref{1.11}),(\ref{3.3})}{=} \exp\Big\{\mbox{\f $\dis\frac{u_N}{d}$} \;N^{d-2} \dsl_{n \ge 2} \;\mbox{\f $\dis\frac{z^n}{a^n_N}$} \big\langle V, (G_N V)^{n-1} 1\big\rangle_{\IL_N}\Big\}.
\end{split}
\end{equation}

\n
By Proposition \ref{prop3.1} and the fact that $\frac{u_N}{d} \;N^{d-2} / a_N^n$ equals $\frac{1}{2}$ when $n = 2$ and is bounded by $\frac{1}{2}$ and converges to zero for $N$ tending to infinity, when $n \ge 3$, we find that
\begin{equation}\label{3.17}
\ov{\IE}[e^{z \langle \wh{\IL}^N,V\rangle}] \underset{N \r \infty}{\longrightarrow} e^{\frac{z^2}{2} \langle V,GV\rangle} \;\; \mbox{for} \; |z| < c_0^{-1}.
\end{equation}

\n
The claim (\ref{3.15}) then follows by similar arguments as below (\ref{3.13}). \end{proof}

\begin{remark}\label{rem3.4} \rm  By (\ref{1.12}) one sees that for any $N \ge 1$
\begin{equation*}
\ov{\IE} \Big[\dsl_{y \in \IL_N} (1 + |y|)^{-(d+1)} L_{Ny,u_N}\Big] < \infty.
\end{equation*}

\n
Thus on a set of full $\ov{\IP}$-measure, $(1 + |y|)^{-(d+1)}$, and hence all functions $V \in \cS(\IR^d)$, are $\wh{\cL}^N$-integrable. Redefining $\wh{\cL}^N$ on a negligible set as (for instance) being equal to zero, we see that $\wh{\cL}^N(\o)$ is a tempered distribution for each $\o \in \ov{\Omega}$, and $\langle \wh{\cL}^N, V\rangle$ a random variable on $(\ov{\Omega}, \ov{\cA}, \ov{\IP})$ for each $V \in \cS(\IR^d)$. Hence, we can also view $\wh{\cL}^N$ as an $\cS'(\IR^d)$-valued random variable on  $(\ov{\Omega}, \ov{\cA}, \ov{\IP})$. \hfill $\square$
\end{remark}

In the next result about convergence in distribution of $\wh{\cL}^N$, the space $\cS'(\IR^d)$ is tacitly endowed with the strong topology (see p.~60 of \cite{GelfVile64}, or p.~5,~6 of \cite{Ito84}). We recall the notation $P^G$ for the law on $\cS'(\IR^d)$ of the massless Gaussian free field (see (\ref{1.15})). The next result states the convergence of $\wh{\cL}_N$ in distribution to the massless Gaussian free field in the high intensity regime.

\begin{corollary}\label{cor3.5} (under (\ref{0.4}))
\begin{equation}\label{3.18}
\mbox{$\wh{\cL}^N$ converges in law to $P^G$ as $N \r \infty$}.
\end{equation}
\end{corollary}

\begin{proof}
By the L\'evy continuity theorem on $\cS'(\IR^d)$ (see Th\'eor\`eme 2, p.~516 of \cite{Meye65}, and also Th\'eor\`eme 3.6.5, p.~69 of \cite{Fern67}), it suffices to show that for each $V \in \cS(\IR^d)$, the variables $\langle \wh{\cL}^N, V\rangle$ converge in distribution to $\langle \Phi, V\rangle$ (under $P^G$), namely a centered Gaussian variable with variance $E(V,V)$, (see (\ref{1.15})).

\medskip
We thus consider $V \in \cS(\IR^d)$ and a continuous $[0,1]$-valued compactly supported function $\chi_L$ on $\IR^d$ equal to $1$ on $C_2$ (see (\ref{3.2}) for notation). We write $\chi_L(\cdot) = \chi(\frac{\cdot}{L})$ for $L \ge 1$ and define $V_L = \chi_L V$.

\medskip
By dominated convergence $E(V_L,V_L) \r E(V,V)$, as $L \r \infty$. By Theorem \ref{theo3.3} our claim (\ref{3.18}) will thus follow once we show that
\begin{equation}\label{3.19}
\lim\limits_{L \r \infty} \;\sup\limits_{N \ge 1} \;\ov{\IE} \big[\big\langle \wh{\cL}^N, V - V_L \big\rangle^2\big] = 0.
\end{equation}

\n
By differentiating, (\ref{3.16}) twice in $z$ at the origin and approximating $V$ by $V_{L'}$, with $L' \r \infty$, we see from Fatou's Lemma that
\begin{equation}\label{3.20}
\ov{\IE} \big[\big\langle \wh{\cL}^N, V - V_L\big\rangle^2\big] \le \big\langle V-V_L, \,G_N(V-V_L)\big\rangle_{L_N}, \; \mbox{for $N \ge 1$, $L \ge 1$}.
\end{equation}
\n
Applying (\ref{3.3}) to $\chi(\cdot) \,|V|(\cdot + a)$ for $a \in \IZ^d$ and using that $\sup\{g_N(y)$; $y \in \IL_N$, $|y|_\infty \ge 1\}$ is uniformly bounded in $N$ (see (\ref{1.7})), we see that $\rho = \sup_{N \ge 1} \|G_N \,|V|\,\|_{L^\infty(\IL_N)} < \infty$, and that
\begin{equation*}
\big\langle V-V_L, \,G_N (V-V_L) \big\rangle_{\IL_N} \le c \, \rho \dsl_{a \in \IZ^d} \;\sup\limits_{y \in a + [0,1)^d} \big|\big(1-\chi_L(y)\big) \,V(y)\big| \underset{L \r \infty}{\longrightarrow} 0.
\end{equation*}

\n
The claim (\ref{3.19}) follows, and this completes the proof of Corollary \ref{cor3.5}.
\end{proof}

\section{Scaling limits via the isomorphism theorem}
\setcounter{equation}{0}

In this short section we revisit Theorem \ref{theo3.3} under the perspective of the isomorphism theorem stated in (\ref{0.6}). This offers a different route to Theorem \ref{theo3.3} for sequences $(u_N)_{N \ge 1}$ in a ``sufficiently high intensity regime''. The special role of dimension $3$ where we recover the full range (\ref{0.4}) will be highlighted. Notation for the Gaussian free field on $\IZ^d$ have been introduced at the end of Section $1$. In particular, $(\varphi_x)_{x \in \IZ^d}$ stands for the canonical field on $\IR^{\IZ^d}$ and $P^g$ for the canonical law of the Gaussian free field.

\medskip
We introduce the random signed measure on $\IR^d$
\begin{equation}\label{4.1}
\Phi^N = \mbox{\f $\dis\frac{1}{\sqrt{d} \,N^{\frac{d}{2} + 1}}$} \;\dsl_{x \in \IZ^d} \varphi_x \,\delta_{\frac{x}{N}}.
\end{equation}

\n
Given a continuous compactly supported function $V$ on $\IR^d$ and a positive sequence $(u_N)_{N \ge 1}$ the identity (\ref{0.6}) implies that under $P^g \otimes \ov{P}$
\begin{equation}\label{4.2}
\begin{array}{l}
\fr \;\dsl_{z \in \IZ^d} V\Big(\mbox{\f $\dis\frac{x}{N}$}\Big) \big(\varphi^2_x - g(0)\big) + \dsl_{x \in \IZ^d} V \Big(\mbox{\f $\dis\frac{x}{N}$}\Big)(L_{x,u_N} - u_N) \stackrel{\rm law}{=}
\\[2ex]
\fr \;\dsl_{z \in \IZ^d} V\Big(\mbox{\f $\dis\frac{x}{N}$}\Big) \big(\varphi^2_x - g(0)\big) + \sqrt{2 u_N} \;\dsl_{x \in \IZ^d} V\Big(\mbox{\f $\dis\frac{x}{N}$}\Big)\, \varphi_x.
\end{array}
\end{equation}

\n
With the notation (\ref{0.5}) and (\ref{4.1}), this identity can be rewritten under the form
\begin{equation}\label{4.3}
\begin{array}{l}
\fr \;\dsl_{z \in \IZ^d} V\Big(\mbox{\f $\dis\frac{x}{N}$}\Big) \big(\varphi^2_x - g(0)\big) + \sqrt{2d \,u_N} \; N^{\frac{d}{2} +1} \big\langle \wh{\cL}^N,V\big\rangle   \stackrel{\rm law}{=}
\\[2ex]
\fr \;\dsl_{z \in \IZ^d} V\Big(\mbox{\f $\dis\frac{x}{N}$}\Big) \big(\varphi^2_x - g(0)\big) + \sqrt{2d \,u_N} \; N^{\frac{d}{2} +1} \big\langle \Phi^N,V\big\rangle .
\end{array}
\end{equation}

\n
Importantly, note that $\wh{\cL}^N$ involves $u_N$ in its definition, but $\Phi^N$ does not (cf.~(\ref{4.1})).

\medskip
In the next lemma we look at the size of the terms in (\ref{4.3}) which involve the Gaussian free field. To this end, we introduce the sequence
\begin{equation}\label{4.4}
b_N = \left\{ \begin{array}{ll}
N^4 & \mbox{when $d=3$},
\\
N^4 \log N & \mbox{when $d=4$},
\\
N^d & \mbox{when $d \ge 5$}.
\end{array}\right.
\end{equation}
\begin{lemma}\label{lem4.1}
\begin{align}
&\mbox{$\big\langle\Phi^N,V\big\rangle$ is a centered Gaussian variable with variance $\big\langle V,G_N V\big\rangle_{\IL_N}$}. \label{4.5}
\\[1ex]
&\mbox{$\big\langle\Phi^N,V\big\rangle$ converges in distribution to a centered Gaussian variable} \label{4.6}
\\[-0.5ex]
&\mbox{with variance $\mbox{\f $\dis\int$} V(y) \,G(y-y') \,V(y') \,dy \,dy'$, as $N \r \infty$}. \nonumber
\\[1ex]
& \mbox{\f $\dis\frac{1}{b_N}$} \;E^{P^g} \Big[\Big(\dsl_{x \in \IZ^d} V\Big(\mbox{\f $\dis\frac{x}{N}$}\Big) \big(\varphi^2_x - g(0)\big)\Big)^2\Big] \; \mbox{has a positive limit when $N \r \infty$}. \label{4.7}
\end{align}
\end{lemma}

\begin{proof}
We begin with (\ref{4.5}) and note that $\langle \Phi^N,V\rangle$ is a centered Gaussian variable with variance $\frac{1}{dN^{d+2}} \;\sum_{x,x'\in \IZ^d} V(\frac{x}{N})\,g(x,x') \,V(\frac{x'}{N}) \underset{(\ref{1.8})}{\stackrel{(\ref{1.5}),(\ref{1.6})}{=}} \langle V,G_N V\rangle_{\IL_N}$, whence the claim.

\medskip
Next, (\ref{4.6}) is an immediate consequence of (\ref{4.5}) and (\ref{3.4}) (with $n=2$). We then turn to (\ref{4.7}). We note that from Lemma 5.2.6, p.~201 of \cite{MarcRose06}, one has the identity 
\begin{equation}\label{4.8}
E^{P^g} \big[\big(\varphi^2_x - g(0)\big) \big(\varphi^2_{x'} - g(0)\big)\big] = 2 g^2(x,x'), \;\mbox{for $x,x' \in \IZ^d$}.
\end{equation}
As a result we find that
\begin{equation}\label{4.9}
\mbox{\f $\dis\frac{1}{b_N}$} \;E^{P^g} \Big[\Big(\dsl_{x \in \IZ^d} V\Big(\mbox{\f $\dis\frac{x}{N}$}\Big) \big(\varphi^2_x - g(0)\big)\Big)^2\Big]  = \mbox{\f $\dis\frac{2}{b_N}$} \;\dsl_{x,x' \in \IZ^d} V\Big(\mbox{\f $\dis\frac{x}{N}$}\Big)\,g^2(x-x') \,V \Big(\mbox{\f $\dis\frac{x'}{N}$}\Big).
\end{equation}

\n
When $d=3$, $G(y) = c\,|y|^{-1}$ (see (\ref{1.3})) is locally square integrable on $\IR^3$, and the above quantity equals (see (\ref{1.6}))
\begin{equation}\label{4.10}
\mbox{\f $\dis\frac{18}{N^6}$} \;\dsl_{y,y' \in \IL_N} V(y)\,g^2_N(y-y')\,V(y') \underset{N \r \infty}{\longrightarrow} 18 \dis\int V(y) \,G^2(y-y')\,V(y') \,dy \, dy',
\end{equation}

\n
as can be seen by separately considering the terms $|y-y'| \ge \gamma$ and $|y - y'| < \gamma$, with $y,y' \in \IL_N$, in the spirit of what was done below (\ref{3.8}), letting first $N$ go to infinity, and then $\gamma$ go to zero. Note that the integral in (\ref{4.10}) is equal to (with $p_t(\cdot,\cdot)$ the Brownian transition density, see above (\ref{2.1})):
\begin{equation*}
\dis\int_{\IR_+ \times \IR_+ \times \IR^d \times \IR^d} ds \,dt \,dz\,dz' \Big(\dis\int_{\IR^d} dy \,V(y) \,p_{\frac{s}{2}}(y,z) \,p_{\frac{t}{2}}(y,z')\Big)^2 
\end{equation*}

\n
by a similar calculation as in Proposition 4.8, p.~75 of \cite{Szni98a}. This quantity is positive when $V$ is not identically equal to zero, see the bottom of p.~75 of \cite{Szni98a} for a similar argument.

\medskip
When $d=4$, $b_N = N^4 \log N$, and given $L \ge 1$ and $\gamma > 0$, the contribution in the right-hand side of (\ref{4.9}) of the terms with $|x-x'| \le L$ or $|x-x'| \ge \gamma N$ is $O(\frac{1}{\log N})$, as $N$ tends to infinity, by (\ref{1.3}). Combined with the fact that $\frac{1}{\log N} \;\sum_{0 < |x| \le N} \;\frac{1}{|x|^4}$ tends to a positive limit, when $N$ tends to infinity (as can be seen by looking at the difference of the sum and the integral $\int_{B(0,N) \backslash B(0,1)} \frac{dx}{|x|^4}$), we find that when $d=4$,
\begin{equation}\label{4.11}
\mbox{\f $\dis\frac{2}{N^4 \log N}$} \;\dsl_{x,x' \in \IZ^4} \,V\Big(\mbox{\f $\dis\frac{x}{N}$}\Big) \,g^2(x-x') \,V\Big(\mbox{\f $\dis\frac{x'}{N}$}\Big) \underset{N}{\longrightarrow} c \dis\int_{\IR^4} V^2(y) \,dy.
\end{equation} 
Finally, when $d \ge 5$, the sum $\sum_{x \in \IZ^d} g^2(x)$ converges by (\ref{1.3}), and hence
\begin{equation}\label{4.12}
\mbox{\f $\dis\frac{2}{N^d}$} \;\dsl_{x,x' \in \IZ^4} \,V\Big(\mbox{\f $\dis\frac{x}{N}$}\Big) \,g^2(x-x') \,V\Big(\mbox{\f $\dis\frac{x'}{N}$}\Big) \underset{N}{\longrightarrow} c \dis\int_{\IR^d} V^2(y) \,dy.
\end{equation} 

\n
We have thus shown (\ref{4.7}), and this concludes the proof of Lemma \ref{lem4.1}.
\end{proof}

We will now conclude this short section with some remarks, in particular linking together (\ref{4.3}), Lemma \ref{lem4.1}, and Theorem \ref{theo3.3}.

\begin{remark}\label{rem4.2} \rm ~

\medskip\n
1) By similar arguments as in Remark \ref{rem3.4} we can view $\Phi^N$ as an $\cS'(\IR^d)$-valued random variable. The same proof as in Corollary \ref{cor3.5} yields that
\begin{equation}\label{4.13}
\mbox{$\Phi^N$ converges in law to $P^G$, as $N \r \infty$}.
\end{equation}

\medskip\n
2) When $u_N \,N^{d+2}/b_N \r \infty$, we can divide both members of (\ref{4.3}) by $\sqrt{2d \,u_N} \; N^{\frac{d}{2} + 1}$. By Lemma \ref{lem4.1} we then see that for each continuous compactly supported function $V$ on $\IR^d$, both $\langle \wh{\cL}^N,V\rangle$ and $\langle \Phi^N,V\rangle$ converge in distribution to a centered Gaussian variable with variance $\langle V, GV\rangle$ (once again, note that $u_N$ enters the definition of $\wh{\cL}^N$ but not that of $\Phi^N$). One thus recovers the convergence asserted in Theorem \ref{theo3.3} under the assumption
\begin{equation}\label{4.14}
\left\{ \begin{array}{ll}
u_N \, N \r \infty, & \mbox{when $d=3$},
\\[1ex]
u_N \, \mbox{\f $\dis\frac{N^2}{\log N}$} \r \infty, & \mbox{when $d=4$},
\\[2ex]
u_N \,N^2 \r \infty, & \mbox{when $d \ge 5$}.
\end{array}\right.
\end{equation}

\n
This recovers the full range of validity of Theorem \ref{theo3.3} (i.e.~$u_N \,N^{d-2} \r \infty$), when $d=3$, but only part of the range when $d \ge 4$. Observe that when $d \ge 4$, in the high intensity regime $u_N \,N^{d-2} \r \infty$, $\langle \wh{\cL}^N,V\rangle$ and $\langle \Phi^N,V\rangle$ have the same distributional limit, regardless of whether (\ref{4.14}) breaks down or not.

\bigskip\n
3) One may wonder from the above discussion, whether some cancellations may be extracted from (\ref{4.3}). In the high intensity regime $u_N \,N^{d-2} \r \infty$, given $M \ge 1$, can one for each $N$ couple two Gaussian free fields $(\varphi_x)_{x \in \IZ^d}$, $(\psi_x)_{x \in \IZ^d}$, with $(L_{x,u_N})_{x \in \IZ^d}$, so that
\begin{align*}
&\mbox{$(\varphi_x)_{x \in \IZ^d}$ is independent of $(L_{x,u_N})_{x \in \IZ^d}$},
\\[1ex]
&\fr \;\varphi^2_x + L_{x,u_N} = \fr \big(\psi_x + \sqrt{2u_N}\big)^2, \;\mbox{a.s., for $x \in NC_M$ (see (\ref{3.2}) for notation)},
\intertext{and for any continuous function $V$ supported in $C_M$,}
&\fr \;\dsl_{x \in \IZ^d} V\Big(\mbox{\f $\dis\frac{x}{N}$}\Big) (\varphi^2_x - \psi^2_x) \big/ \big(\sqrt{u_N} \; N^{\frac{d}{2} +1}\big) \underset{N}{\longrightarrow} 0 \;\mbox{in probability?}
\end{align*}

\n
Such a coupling would yield a simple way to recover the identity of the distributional limits of $\langle \wh{\cL}^N,V\rangle$ and $\langle \Phi^N,V\rangle$ in the high intensity regime. Of course, when (\ref{4.14}) holds such a coupling can be achieved (by (\ref{4.7}) the last  condition is automatically satisfied). \hfill $\square$
\end{remark}

\section{The special case of dimension 3}
\setcounter{equation}{0}

We have seen in the last section that when $d=3$ the isomorphism theorem (\ref{0.6}) offers a quick route to the study of the asymptotic behaviour of $\wh{\cL}^N$ in the full range of the high intensity regime (\ref{0.4}).  In this section we investigate what happens ``at the edge'', in the constant intensity regime (\ref{0.2}). The scaling limit of the distributional identity (\ref{0.6}) (with proper counter terms) will bring as a by-product an isomorphism theorem for $3$-dimensional Brownian interlacements, see Theorem \ref{theo5.1} and Corollary \ref{cor5.3}. This last corollary has a similar flavor to the distributional identity derived by Le Jan in the context of a Poisson gas of Brownian loops at half-integer intensity, see Chapter 10 \S2, p.~104 of \cite{Leja12}. We recall the notation of the end of Section 1 concerning Gaussian free fields. Throughout this section $d=3$.

\medskip
We denote by $H$ the Gaussian space, which is the $L^2(P^G)$-closure of $\big\{\big\langle\Phi,f\big\rangle; \,f \in \cS(\IR^d)\big\}$. For $Y,Z$ in $H$ the Wick product, see \cite{Jans97}, p.~23, 24, $\mbox{{\large :}} YZ\mbox{{\large :}}$ equals $YZ - E^{P^G}[YZ]$, and one has the identity, see \cite{Jans97}, p.~11, 12 or \cite{MarcRose06}, p.~201:
\begin{equation}\label{5.1}
\begin{split}
E^{P^G}[\mbox{{\large :}} Y^2 \mbox{{\large :}} \;\;\mbox{{\large :}} Z^2\mbox{{\large :}}] & = 2 E^P[YZ]^2, \;\;\mbox{for $Y,Z \in H$}
\\ 
& = 2 \big\langle f,G h\big\rangle^2, \;\mbox{when $Y = \big\langle \Phi,f\big\rangle, \;Z= \big\langle\Phi, h\big\rangle$, with $f,h \in \cS(\IR^d)$}.
\end{split}
\end{equation}

\n
When $d = 3$, the Green function $G(y,y')$ is locally square integrable, and for bounded measurable functions $V$ on $\IR^3$ vanishing outside a compact set, one can define
\begin{equation}\label{5.2}
\dis\int V(y)\, \mbox{{\large :}}  \Phi^2_y \mbox{{\large :}} \,dy = \lim\limits_{\ve \r 0} \;\dis\int V(y) \,\mbox{{\large :}} \Phi^2_{y,\ve} \mbox{{\large :}} \,dy \;\;\mbox{in} \; L^2(P^G),
\end{equation}
where $\Phi_{y,\ve} = \big\langle \Phi, \rho_{y,\ve}\big\rangle$, with $\rho_{y,\ve}(\cdot) = \frac{1}{\ve^3} \,\rho(\frac{\cdot - y}{\ve})$, and $\rho(\cdot)$ a non-negative, smooth, compactly supported function on $\IR^3$ such that $\int \rho(z)\,dz = 1$. This fact uses (\ref{5.1}), the crucial local square integrability of $G(\cdot)$, and similar arguments as in the proof of Proposition 8.5.1, p.~153 of \cite{GlimJaff81}, see also \cite{Leja12}, p.~101. In addition, the limit object in the left-hand side of (\ref{5.1}) does not depend on the specific choice of $\rho(\cdot)$. For convenience, we will assume that $\rho(\cdot)$ is radially symmetric and supported in $B(0,1)$.

\medskip
For $V$ as above, one also defines (in a simpler fashion) the element of $H$:
\begin{equation}\label{5.3}
\dis\int V(y) \,\Phi_y \,dy = \lim\limits_{\ve \r 0} \dis\int V(y) \Phi_{y,\ve} \,dy \; \mbox{in} \;L^2(P^G).
\end{equation}

\n
We now assume that we are in the constant intensity regime (\ref{0.2}), and for $V$ a continuous compactly supported function on $\IR^3$, we rewrite (\ref{4.2}) as the following identity in distribution under $P^g \otimes \ov{\IP}$:
\begin{equation}\label{5.4}
\fr \; \dsl_{x \in \IZ^3} V\Big(\mbox{\f $\dis\frac{x}{N}$}\Big) \,\mbox{{\large :}} \varphi_x^2 \mbox{{\large :}} + dN^2 \big\langle \cL^N,V\big\rangle \stackrel{\rm law}{=} \fr \dsl_{x \in \IZ^3} V\Big(\mbox{\f $\dis\frac{x}{N}$}\Big)\, \mbox{{\large :}} \big(\varphi_x + \sqrt{2 u_N}\big)^2\mbox{{\large :}} \,,
\end{equation} 

\n
where $\mbox{{\large :}}(\varphi_x + \sqrt{2u_N})^2 \mbox{{\large :}} \; = \mbox{{\large :}} \varphi^2_x \mbox{{\large :}} \; + 2 \sqrt{2u_N} \;\varphi_x + 2u_N$ and $\mbox{{\large :}} \varphi^2_x\mbox{{\large :}}  = \varphi^2_x - g(0)$.

\medskip
We already know that $\langle \cL^N,V\rangle$ converges in distribution to $\langle \cL_\alpha,V\rangle$ (see Theorem \ref{theo3.2}). The limit behavior of the other terms in (\ref{5.4}) is described by

\bigskip
\begin{theorem}\label{theo5.1} ($d=3$, $V$ continuous with compact support)

\medskip
For $\alpha \ge 0$ and $u_N = 3 \alpha/N$ ($= d \alpha/N^{d-2}$), as $N \r \infty$,
\begin{equation}\label{5.5}
\mbox{\f $\dis\frac{1}{3N^2}$} \;\dsl_{x \in \IZ^3} \;V\Big(\mbox{\f $\dis\frac{x}{N}$}\Big)\;\mbox{{\rm \large :}}\big(\varphi_x + \sqrt{2u_N}\big)^2\mbox{{\rm \large :}} \; \mbox{converges in law to} \; \dis\int V(y)\, \mbox{{\rm \large :}} (\Phi_y + \sqrt{2 \alpha})^2\mbox{{\rm \large :}}\, dy
\end{equation} 

\n
(the last term can be defined as $\int V(y)$ $\mbox{{\rm \large :}}\Phi^2_y\mbox{{\rm \large :}}\,dy + 2 \sqrt{\alpha} \int V(y) \, \Phi_y \,dy + 2 \alpha \int V(y)\,dy$, but see also (\ref{5.10}) below).

\medskip
Under $P^G \otimes \IP$ one has the distributional identity
\begin{equation}\label{5.6}
\fr \dis\int V(y) \,\mbox{{\rm \large :}} \Phi^2_y \mbox{{\rm \large :}} \, dy + \big\langle \cL_\alpha,V \big\rangle \stackrel{\rm law}{=} \fr \dis\int V(y) \,\mbox{{\rm \large :}}\big(\Phi_y + \sqrt{2\alpha}\big)^2\mbox{{\rm \large :}} \, dy.
\end{equation}
\end{theorem}

\begin{proof}
We first observe that (\ref{5.5}), (\ref{5.4}) (choosing $\alpha = 0$ to handle the first term of (\ref{5.4})) and Theorem \ref{theo3.2} readily imply (\ref{5.6}). Thus, we only need to prove (\ref{5.5}). This is a case of lattice field approximation (see chapter 9 \S 5 and \S 6 of \cite{GlimJaff81}, also Chapter 8 of \cite{Simo74}). However, the fact that we consider the $d=3$ situation and massless free fields, makes the set-up a bit different. Since some care is needed, see for instance \cite{GlimJaff81}, p.~185-187, we sketch the proof for the reader's convenience. We define
\begin{equation}\label{5.7}
\varphi_{y,N} = \sqrt{\mbox{\f $\dis\frac{N}{d}$}} \;\varphi_{Ny} \;\; \mbox{(with $d=3$), for $y \in \IL_N$},
\end{equation}
and note that the first  term of (\ref{5.5}) equals
\begin{equation}\label{5.8}
\mbox{\f $\dis\frac{1}{N^3}$} \;\dsl_{y \in \IL_N} V(y) \,\mbox{{\rm \large :}}\big(\varphi_{y,N} + \sqrt{2 \alpha}\big)^2\mbox{{\rm \large :}}\,.
\end{equation}
We now choose $\ve(N) > 0$ tending to $0$ not too fast so that
\begin{equation}\label{5.9}
\ve(N) \r 0 \; \mbox{and} \; N \,\ve^3(N) \r \infty, \;\; \mbox{as $N \r \infty$}.
\end{equation}
As we explain below, the claim (\ref{5.5}) will follow once we establish the following three facts:
\begin{align}
&\lim\limits_N \,\Big\| \dis\int V(y)\, \mbox{{\rm \large :}} \big(\Phi_y + \sqrt{2 \alpha}\big)^2 \mbox{{\rm \large :}}\, dy - \dis\int V(y) \,\mbox{{\rm \large :}} \big(\Phi_{y,\ve} + \sqrt{2 \alpha}\big)^2 \mbox{{\rm \large :}}\,dy\Big\|_{L^2(P^G)} = 0, \label{5.10}
\\[1ex]
&\lim\limits_N \,\Big\| \dis\int V(y)\, \mbox{{\rm \large :}} \big(\Phi_{y,\ve} + \sqrt{2 \alpha}\big)^2 \mbox{{\rm \large :}}\, dy - \mbox{\f $\dis\frac{1}{N^3}$} \; \dsl_{y \in \IL_N} V(y) \,\mbox{{\rm \large :}} \big(\Phi_{y,\ve} + \sqrt{2 \alpha}\big)^2 \mbox{{\rm \large :}}\,dy\Big\|_{L^1(P^G)} = 0, \label{5.11}
\end{align}
and for small real $z$, i.e.~$|z| \le r_0$,
\begin{equation}\label{5.12}
\begin{split}
\lim\limits_N &\;\Big(E^{P^G}\Big[\exp\Big\{\mbox{\f $\dis\frac{z}{N^3}$} \; \dsl_{y \in \IL_N} V(y) \, \mbox{{\rm \large :}} \big(\Phi_{y,\ve} + \sqrt{2 \alpha}\big)^2\, \mbox{{\rm \large :}}\Big\}\Big]  
\\
&-\; E^{P^g} \,\Big[\exp\Big\{\mbox{\f $\dis\frac{z}{N^3}$} \; \dsl_{y \in \IL_N} V(y) \, \mbox{{\rm \large :}} \big(\varphi_{y,N} + \sqrt{2 \alpha}\big)^2\mbox{{\rm \large :}}\Big\}\Big]\Big) = 0
\end{split}
\end{equation}

\n
We first explain how these three facts yield a proof of (\ref{5.5}). By Theorem 3.50, p.~39 of \cite{Jans97}, we can replace $L^1(P^G)$ by $L^2(P^G)$ in (\ref{5.11}). Together with (\ref{5.10}) we find that
\begin{equation}\label{5.13}
\lim\limits_N \;\Big\|\mbox{\f $\dis\frac{1}{N^3}$} \;\dsl_{y \in \IL_N} V(y) \,\mbox{{\rm \large :}}\big(\Phi_{y,\ve} + \sqrt{2 \alpha}\big)^2\mbox{{\rm \large :}} - \dis\int V(y) \,\mbox{{\rm \large :}}\big(\Phi_y + \sqrt{2 \alpha}\big)^2 \mbox{{\rm \large :}}\,dy\Big\|_{L^2(P^G)} = 0.
\end{equation}
By Theorem 6.7, p.~82 of \cite{Jans97}, we thus see that for some $r_1 > 0$,
\begin{equation}\label{5.14}
\sup\limits_{N \ge 1} \;E^{P^G} \Big[\cosh \Big(\mbox{\f $\dis\frac{r_1}{N^3}$} \;\dsl_{y \in \IL_N} V(y)\,\mbox{{\rm \large :}} \big(\Phi_{y,\ve} + \sqrt{2 \alpha}\big)^2\mbox{{\rm \large :}}\Big)\Big] < \infty.
\end{equation}
By (\ref{5.12}), we then find that for some $0 < r_2 < r_1 \wedge r_0$,
\begin{equation}\label{5.15}
\sup\limits_{N \ge 1} \;E^{P^g} \Big[\cosh \Big(\mbox{\f $\dis\frac{r_2}{N^3}$} \;\dsl_{y \in \IL_N} V(y)\,\mbox{{\rm \large :}} \big(\varphi_{y,N} + \sqrt{2 \alpha}\big)^2\mbox{{\rm \large :}}\Big)\Big] < \infty.
\end{equation}

\n
Thus, by a similar argument as below (\ref{3.12}), we see that the laws under $P^g$ of the random variables $\frac{1}{N^3} \sum_{y \in \IL_N} V(y)\, \mbox{{\rm \large :}} (\varphi_{y,N} + \sqrt{2 \alpha})^2\mbox{{\rm \large :}}$ are tight, and, when along a subsequence $N_k$ they converge in distribution to a random variable $U$, then, for $|z| < r_2$, $e^{zU}$ is integrable and one has
\begin{equation}\label{5.16}
\begin{split}
E[e^{zU}] & = \lim\limits_{N_k} \;E^{P^g} \Big[\exp\Big\{ \mbox{\f $\dis\frac{z}{N^3}$} \;\dsl_{y \in \IL_N} V(y) \,\mbox{{\rm \large :}} \big(\varphi_{y,N} + \sqrt{2 \alpha}\big)^2\mbox{{\rm \large :}}\Big\}\Big] 
\\ 
&\!\!\! \stackrel{(\ref{5.12})}{=}  \lim\limits_{N_k} \;E^{P^G} \Big[\exp\Big\{ \mbox{\f $\dis\frac{z}{N^3}$} \;\dsl_{y \in \IL_N} V(y) \,\mbox{{\rm \large :}} \big(\Phi_{y,\ve} + \sqrt{2 \alpha}\big)^2\mbox{{\rm \large :}}\Big\}\Big] 
\\
& = E^{P^G} \Big[\exp \Big\{ z \dis\int V(y) \,\mbox{{\rm \large :}} \big(\Phi_y + \sqrt{2 \alpha}\big)^2\mbox{{\rm \large :}}\,dy\Big\}\Big] ,
\end{split}
\end{equation}

\n
using in the last step both the uniform integrability of the variables under the expectation in the second line when $|z| < r_2$ by (\ref{5.14}), and (\ref{5.13}) (see Theorem 5.4, p.~32 of \cite{Bill68}). By analyticity, the equality between the first and the last term of (\ref{5.16}) extends to the strip $\{z \in \IC$; $|Rez| < r_2\}$. Hence, the random variable $U$ has same distribution as $\int V(y)\,\mbox{{\rm \large :}}(\Phi_y + \sqrt{2 \alpha})^2\mbox{{\rm \large :}}\,dy$, regardless of the extracted subsequence $N_k$, and (\ref{5.5}) follows.

\medskip
There remains to establish (\ref{5.10}) - (\ref{5.12}). The claim (\ref{5.10}) is an immediate consequence of (\ref{5.2}), (\ref{5.3}) and the identity $\mbox{{\rm \large :}}(\Phi_{y,\ve} + \sqrt{2 \alpha})^2\mbox{{\rm \large :}}\, = \,\mbox{{\rm \large :}}\Phi^2_{y,\ve}\, \mbox{{\rm \large :}} + 2 \sqrt{2\alpha} \;\Phi_{y,\ve} + 2 \alpha$.

\medskip
We then turn to (\ref{5.11}) and simply explain why
\begin{equation}\label{5.17}
\lim\limits_N \;\Big\| \dis\int V(y) \,\mbox{{\large :}}  \Phi^2_{y,\ve}\mbox{{\large :}} \,dy - \mbox{\f $\dis\frac{1}{N^3}$} \;\dsl_{y \in \IL_N} V(y)\,\mbox{{\large :}} \Phi^2_{y,\ve}\mbox{{\large :}}\, \Big\|_{L^1(P^G)} = 0.
\end{equation}

\n
The case of the linear term (where $2 \sqrt{2 \alpha} \;\Phi_{y,\ve}$ replaces $\mbox{{\large :}} \Phi^2_{y,\ve}\mbox{{\large :}}$) is simpler to handle, and the constant terms (after developing the square in (\ref{5.11})) cancel each other. The quantity under the limit in (\ref{5.17}) is bounded above by $I_1 + I_2$, where
\begin{align}
I_1  = &\; E^{P^G} \Big[\Big(\dis\int h_N(y)\,\mbox{{\large :}} \Phi^2_{y,\ve}\mbox{{\large :}}\,dy\Big)^2\Big]^{\frac{1}{2}}, \;\mbox{with} \label{5.18}
\\
& \; h_N(y) = \dsl_{z \in \IL_N} \big(V(y) - V(z)\big) \,1\Big\{y \in z + \mbox{\f $\dis\frac{1}{N}$} \;[0,1)^3\Big\}, \nonumber
\\[2ex]
I_2 = & \dsl_{z \in \IL_N} |V(z)| \,E^{P^G} \Big[\Big| \dis\int_{z + \frac{1}{N} [0,1)^3} \mbox{{\large :}}\Phi^2_{y,\ve} \mbox{{\large :}} \, - \, \mbox{{\large :}}\Phi^2_{z,\ve}\mbox{{\large :}}\,dy\Big|\Big]. \label{5.19}
\end{align}
By (\ref{5.1}) we see that
\begin{align}
& I^2_1 \le 2 \dis\int |h_N(y)| \;|h_N(y)| \,G^2_\ve(y-y')\,dy\,dy', \;\mbox{where} \label{5.20}
\\ 
&G_\ve(y) = E^{P^G}[\Phi_{y,\ve} \,\Phi_{0,\ve}] = \dis\int \rho_\ve(z-y)\,G(z-z')\,\rho_\ve(z')\,dz\,dz', \;\mbox{for $y \in \IR^3$}. \label{5.21}
\end{align}
Note that $\rho_\ve(\cdot)$ is spherically symmetric with support in $B(0,\ve)$ and $G$ harmonic on $\IR^3 \backslash \{0\}$, so that
\begin{equation}\label{5.22}
\begin{split}
G_\ve(y)  = &\; G(y), \qquad \mbox{when $|y| > 2 \ve$}
\\[1ex]
\le & \;\mbox{\f $\dis\frac{c}{\ve}$} \;\|\rho\|_\infty, \quad \mbox{when $|y| \le 2 \ve$}, 
\end{split}
\end{equation}

\n
where in the last step we used that $\int_{|z-z'|\le r}G(z-z') \,dz = c\,r^2$, for $r \ge 0$, $z' \in \IR^3$.

\medskip
We introduce a box $\Lambda = [-M,M)^3$, with $M \ge 1$ integer, such that
\begin{equation}\label{5.23}
\mbox{$\Lambda$ contains all points within $|\cdot |_\infty$-distance $1$ from the support of $V$.}
\end{equation}

\n
Note that $h_N$ vanishes outside $\Lambda$ (cf.~(\ref{5.18})), and converges uniformly to $0$ as $N \r \infty$. Moreover, by (\ref{5.22}) and the local square integrability of $G^2$, we see that $\int_{\Lambda \times \Lambda} G^2_\ve(y-y')\,dy \,dy \r \int_{\Lambda \times \Lambda} G^2(y-y') \,dy\,dy'$, as $N \r \infty$. Hence, see (\ref{5.20}),
\begin{equation}\label{5.24}
I_1 \r 0, \;\; \mbox{as $N \r \infty$}.
\end{equation}

\n
We then control $I_2$. We first note that $N^3 | \Lambda |$ bounds the number of boxes $z + \frac{1}{N} [0,1)^3$ intersecting the support of $V$ (with $|\Lambda |$ the volume of $\Lambda$, and $ z \in \IL_N$). Hence,
\begin{equation}\label{5.25}
I_2 \le \|V\|_\infty \,|\Lambda | \;\sup\limits_{y \in \frac{1}{N}[0,1)^3} E^{P^G} [ | \mbox{{\large :}} \Phi^2_{y,\ve} \mbox{{\large :}} \, - \, \mbox{{\large :}}\Phi^2_{0,\ve}\mbox{{\large :}}|].
\end{equation}
Moreover, for $y \in \IR^3$ we have
\begin{equation}\label{5.26}
\begin{split}
E^{P^G} [(\mbox{{\large :}} \Phi^2_{y,\ve} \mbox{{\large :}} \, - \, \mbox{{\large :}}\Phi^2_{0,\ve}\mbox{{\large :}})^2] & = E^{P^G} [\mbox{{\large :}} \Phi^2_{y,\ve} \mbox{{\large :}} ] + E^{P^G} [\mbox{{\large :}} \Phi^2_{0,\ve} \mbox{{\large :}} ] - 2 E^{P^G} [\mbox{{\large :}} \Phi^2_{y,\ve} \mbox{{\large :}} \;\mbox{{\large :}} \Phi^2_{0,\ve} \mbox{{\large :}} ]
\\[1ex]
& \!\! \stackrel{(\ref{5.1})}{=} 4\big(G^2_\ve(0) - G^2_\ve(y)\big) \stackrel{(\ref{5.22})}{\le} \mbox{\f $\dis\frac{c (\rho)}{\ve}$}  \big(G_\ve(0) - G_\ve(y)\big)
\\[1ex]
& \le \mbox{\f $\dis\frac{c (\rho)}{\ve}$} \; \mbox{\f $\dis\frac{|y|}{\ve^2}$}  \;\; \mbox{(since $\| \,|\nabla G_\ve| \,\|_{L^\infty(\IR^3)} \le \mbox{\f $\dis\frac{c (\rho)}{\ve^2}$} \big)$}.
\end{split}
\end{equation}
We thus see that
\begin{equation}\label{5.27}
I_2 \le c(\rho) \,\|V\|_\infty |\Lambda| \;(N \ve^3)^{-\frac{1}{2}} \underset{N}{\longrightarrow} 0 \;\; \mbox{(by (\ref{5.9}))}.
\end{equation}
This complete the proof of (\ref{5.11}). We now turn to (\ref{5.12}).

\medskip
By (\ref{5.22}) and (\ref{1.2}), (\ref{1.3}), (\ref{1.6}) we have for $y,y' \in \IL_N, N \ge 1$,
\begin{equation}\label{5.28}
\begin{array}{l}
G_\ve(y-y') \le c(\rho) \Big(|y - y'| \vee \mbox{\f $\dis\frac{1}{N}$}\Big)^{-1} \; \mbox{and}
\\[2ex]
g_N(y,y') \le c\Big(|y-y'| \vee \mbox{\f $\dis\frac{1}{N}$}\Big)^{-1}.
\end{array}
\end{equation}
We can thus make sure that for small real $z$,
\begin{align}
& \sup\limits_{N,y \in \IL_N} \;\mbox{\f $\dis\frac{|z|}{N^3}$} \; \dsl_{y \in \IL_N} G_\ve(y-y') \,|V(y')| \le \fr, \;\mbox{and} \label{5.29}
\\[1ex]
&\sup\limits_{N,y \in \IL_N} \;\mbox{\f $\dis\frac{|z|}{N^3}$} \; \dsl_{y \in \IL_N} g_N(y-y') \,|V(y')| \le \fr \label{5.30}
\end{align}

\n
(incidentally, by (\ref{3.6}) and the continuity of $G\,|V|$, we also have $|z| \, \|G\,|V|\|_{L^\infty(\IR^3)} \le \frac{1}{2}$).

\pagebreak
If $\Lambda_N = \Lambda \cap \IL_N$, the matrix $\big( G_\ve(y-y'))_{y,y' \in \Lambda_N}$ is positive definite. Indeed, by (\ref{5.21}) it suffices to note that $\rho_\ve(\cdot-y), y \in \Lambda_N$, are linearly independent (one easily sees by taking Fourier transforms and using an analytic continuation argument that any linear relation between these functions is trivial).

By Lemma 5.2.1 on p.~198 of \cite{MarcRose06}, or Proposition 2.14 on p.~47 of \cite{Szni11d}, for $z$ as above, 
\begin{equation}\label{5.31}
E^{P^G} \Big[\exp\Big\{ \mbox{\f $\dis\frac{z}{2N^3}$} \; \dsl_{y \in \IL_N} V(y)\, \mbox{{\large :}} \big(\Phi_{y,\ve} + \sqrt{2\alpha}\big)^2\mbox{{\large :}}\Big\}\Big] = \dis\frac{e^{\alpha z\langle V,(I-z \wt{G}_\ve V)^{-1} 1\rangle_{\IL_N}}}{({\rm det}_2((I-z \wt{G}_\ve V)_{|\Lambda_N \times \Lambda_N}))^{\frac{1}{2}}} \;,
\end{equation}

\n
where ${\rm det}_2 \big((I-z \wt{G}_\ve V)_{|\Lambda_N \times \Lambda_N}\big)$ denotes the regularized determinant of the finite matrix $K$, which is the restriction to $\Lambda_N \times \Lambda_N$ of $\big(\delta_{y,y'} - z G_\ve(y-y') \frac{V(y')}{N^3}\big)_{y,y' \in \IL_N}$, that is $({\rm det} \,K)e^{-Tr(K-I)}$ (where $Tr$ stands for the trace), see \cite{Simo79}, p.~75. In addition, with some abuse of notation, in the term in the exponential in the right-hand side of (\ref{5.31}), $\wt{G}_\ve$ stands for the linear operator on functions on $\IL_N$ defined by a similar formulas as (\ref{1.8}), with $g_N(y,y')$ replaced by $G_\ve(y-y')$, so that $\wt{G}_\ve V$ operates on bounded functions on $\IL_N$ and $I-z \wt{G}_\ve V$ is invertible by (\ref{5.29}). Also, compared to the above references, we used the identity following from the Neumann expansion of $(I-z \wt{G}_\ve V)^{-1}$:
\begin{equation*}
\big\langle V,(I-z \wt{G}_\ve V)^{-1} 1\big\rangle_{\IL_N} = \big\langle V,1\big\rangle_{\IL_N} + z \big\langle V, (I-z \wt{G}_\ve V)^{-1} \wt{G}_\ve V \,1\big\rangle_{\IL_N} .
\end{equation*}

\n
In a similar manner, with analogous notation, for $z$ as above, we have
\begin{equation}\label{5.32}
E^{P^g} \Big[\exp\Big\{ \mbox{\f $\dis\frac{z}{2N^3}$} \; \dsl_{y \in \IL_N} V(y)\, \mbox{{\large :}} \big(\Phi_{y,N} + \sqrt{2\alpha}\big)^2\mbox{{\large :}}\Big\}\Big] = \dis\frac{e^{\alpha z\langle V,(I-z G_N V)^{-1} 1\rangle_{\IL_N}}}{({\rm det}_2((I-z G_N V)_{|\Lambda_N \times \Lambda_N}))^{\frac{1}{2}}} \;,
\end{equation}

\n
By Proposition \ref{prop3.1} and (\ref{5.30}) we have for $z$ as above
\begin{equation}\label{5.33}
\lim\limits_N \;\big\langle V,(I-z G_N V)^{-1} 1\big\rangle_{\IL_N} = \big\langle V,(I-z GV)^{-1} 1\big\rangle .
\end{equation}

\n
By an analogous argument as for the proof of Proposition \ref{prop3.1} and (\ref{5.29}) we have
\begin{equation}\label{5.34}
\lim\limits_N \;\big\langle V,(I-z \wt{G}_\ve V)^{-1} 1\big\rangle_{\IL_N} = \big\langle V,(I-z GV)^{-1}  1\big\rangle .
\end{equation}

\n
As we now explain, the claim (\ref{5.12}) will now follow once we prove that
\begin{align}
\sup\limits_N & \dsl_{y,y' \in \Lambda_N} \big(G_\ve(y,y')^2 + g_N(y,y')^2\big) \;\mbox{\f $\dis\frac{1}{N^6}$} < \infty, \label{5.35}
\\[1ex]
\lim\limits_N & \dsl_{y,y' \in \Lambda_N} \big(G_\ve(y,y') - g_N(y,y')\big)^2 \;\mbox{\f $\dis\frac{1}{N^6}$} = 0. \label{5.36}
\end{align}

\n
Indeed from Theorem 9.2, on p.~75 of \cite{Simo05}, one knows that for any $n \times n$ matrices $A,B$ (and $\| \cdot \|_2$ denotes the Hilbert-Schmidt norm)
\begin{equation}\label{5.37}
|{\rm det}_2 (I+A) - {\rm det}_2(I+B)| \le \|A - B\|_2 \exp\{c(\|A\|_2 + \|B\|_2 + 1)^2\}
\end{equation}
(the constant $c$ does not depend on $n$).

\medskip
Assuming (\ref{5.35}), (\ref{5.36}) one thus finds that for small $z$
\begin{align}
&\lim\limits_N \;{\rm det}_2\big((I-z \wt{G}_\ve V)_{|\Lambda_N \times \Lambda_N}\big) - {\rm det}_2 \big((I-z G_N V)_{|\Lambda_N \times \Lambda_N}\big) = 0, \; \mbox{and} \label{5.38}
\\[1ex]
&\mbox{the regularized determinants in (\ref{5.38}) each belong to $[\frac{1}{2}, \frac{3}{2}]$ for $N \ge 1$} \label{5.39}
\end{align}
(we use (\ref{5.37}) with $B=0$ and ${\rm det}_2(I) = 1$ for this last fact).

\pagebreak
Combined with (\ref{5.33}), (\ref{5.34}) the claim (\ref{5.12}) then follows.

\medskip
There remains to prove (\ref{5.35}) and (\ref{5.36}). The bound (\ref{5.35}) follows from (\ref{5.28}) in a straightforward fashion, and we concentrate on (\ref{5.36}). We pick $\gamma > 0$. The sum in (\ref{5.36}) equals $J_1 + J_2$, where $J_1$ collect the terms where $|y-y'| > \gamma$ and $I_2$ the terms where $|y-y'| \le \gamma$. When $N$ is large, $G_\ve(y-y') = G(y-y')$ for all $|y-y'| > \gamma$ (see (\ref{5.22})), $G(y-y') - g_N(y-y')$ converges uniformly to zero on this set by (\ref{1.7}). We thus find that $\lim_N J_1 = 0$. As for the second term $J_2$, by (\ref{5.28}) we have
\begin{equation}\label{5.40}
\overline{\lim\limits_N} \; J_2 \le \ov{\lim\limits_N} \;\mbox{\f $\dis\frac{c(\Lambda,\rho)}{N^3}$} \;\Big(N^2 + \dsl_{z \in \IL_N, 0 < |z| \le \gamma} \; \mbox{\f $\dis\frac{1}{|z|^2}$}\Big) \le c'(\Lambda, \rho) \,\gamma .
\end{equation}
Letting $\gamma$ tend to $0$, we obtain (\ref{5.36}). This completes the proof of (\ref{5.12}) and hence of Theorem \ref{theo5.1}.
\end{proof} 

\begin{remark}\label{rem5.2} \rm  ~  

\medskip\n
1) Note that for $V$ continuous with compact support on $\IR^3$ one has
\begin{equation*}
\begin{array}{l}
\IE\Big[\Big(\dis\int V(y) \,\mbox{{\large :}} \Phi^2_y\mbox{{\large :}} \,dy\Big)^2\Big] = 2 \dis\int V(y)\,G^2(y-y')\,V(y') \,dy \, dy' \le
\\[1ex]
c\,\|V\|_{L^1(\IR^d)}\big(\|V\|_{L^\infty(\IR^d)} + \|V\|_{L^1(\IR^3)}\big),
\end{array}
\end{equation*}

\medskip\n
and a similar bound holds for $\IE [(\int(V(y)\,\Phi_y \,dy)^2] = \int V(y) \,G(y-y') \,V(y')\,dy\,dy'$.

\medskip
We can thus extend the linear map sending $V$ to $\int V(y)\, \mbox{{\large :}} (\Phi_y + \sqrt{2 \alpha}\big)^2\mbox{{\large :}}  \,dy \in L^2(P^G)$ by continuity to $\cS(\IR^3)$ for any $\alpha \ge 0$. By the regularization theorem, see Theorem 2.3.2, p.~24 of \cite{Ito84}, one can find a version denoted by $\mbox{{\large :}} (\Phi + \sqrt{2 \alpha})^2\mbox{{\large :}}$, which is an $\cS'(\IR^3)$-valued random variable defined on the canonical space where $P^G$ is defined such that for any $V \in \cS(\IR^3)$, 
\begin{equation*}
\mbox{$P^G$-a.s., $\big\langle \mbox{{\large :}} \big(\Phi + \sqrt{2 \alpha}\big)^2\mbox{{\large :}} , V\rangle = \dis\int V(y) \,\mbox{{\large :}} \big(\Phi_y + \sqrt{2 \alpha}\big)^2\mbox{{\large :}} \,dy$}.
\end{equation*}

\medskip\n
2) The intensity measure of $\cL_\alpha$ equals $\alpha\,dy$ (see (\ref{2.39})), and by similar considerations as in Remark \ref{rem3.4} we can also view $\cL_\alpha$ (after a possible redefinition on a set of measure $0$) as an $\cS'(\IR^3)$-valued random variable defined on $(\Omega, \cA, \IP)$. \hfill $\square$
\end{remark}
             
As a consequence of Theorem \ref{theo5.1} we obtain an isomorphism theorem for the occupation measure of Brownian interlacements on $\IR^3$, which arises as the scaling limit (with suitable counter terms) of (\ref{0.6}):

\begin{corollary}\label{cor5.3} ($d=3, \alpha \ge 0$)

\medskip
Under $P^G \otimes \IP$ one has the identity in law on $\cS'(\IR^3)$:
\begin{equation}\label{5.41}
\fr \;\mbox{{\rm \large :}} \Phi^2 \mbox{{\rm \large :}}  + \cL_\alpha \stackrel{\rm law}{=} \fr \; \mbox{{\rm \large :}} \big(\Phi + \sqrt{2 \alpha}\big)^2\mbox{{\rm\large :}} 
\end{equation}
\end{corollary}

\begin{proof}
By Remark \ref{rem5.2} we can obtain the identity (\ref{5.6}) for any $V \in \cS(\IR^3)$. Hence under $P^G \otimes \IP$, the random variables $\langle \frac{1}{2} \, \mbox{{\large :}} \Phi^2\mbox{{\large :}}  + \cL_\alpha, V \rangle$ and $\langle \frac{1}{2} \,\mbox{{\large :}} (\Phi + \sqrt{2 \alpha})^2 \mbox{{\large :}},  V\rangle$ have the same distribution for each $V \in \cS(\IR^3)$, and (\ref{5.41}) easily follows. 
\end{proof}

\begin{remark}\label{rem5.4} \rm In the context of Brownian loops on $\IR^d$, $1 \le d \le 3$ (suitably killed when $d = 1,2$), Le Jan, see p.104 of \cite{Leja12}, showed that $\frac{1}{2} \,\mbox{{\large :}} \Phi^2\mbox{{\large :}} $ (or its corresponding object when $d = 1,2$) has the same law as the renormalized occupation field of a Poisson cloud of Brownian loops with intensity parameter $\alpha = \frac{1}{2}$.

\medskip
Observe that here, in (\ref{5.41}), no renormalization on $\cL_\alpha$ is needed (and $\alpha$ varies over $\IR_+)$. Only terms involving the free field need renormalization. \hfill $\square$

\end{remark}

\end{document}